\documentclass{amsart}
\usepackage[mathscr]{eucal}
\usepackage{amssymb, amsmath,array, amscd}
\usepackage[OT2,T1]{fontenc}

\newtheorem{thm}{Theorem}
\newtheorem{defi}{Definition}
\newtheorem{prop}{Proposition}
\newtheorem{lemma}{Lemma}
\newtheorem{claim}{Claim}

\newtheorem{rem}{Remark}
\newtheorem{cor}{Corollary}

\newtheorem{prob}{Problem}
\newtheorem{ex}{Example}

\def\M{\mathcal{M}}
\def\wt{\widetilde}

\def\de{\mathcal{D}}

\def\fol{\mathbb{F}\text{ol}}
\def\E{\mathcal{E}}

\def\Ga{\Gamma}

\def\p{\mathbb{P}}

\def\Te{\Theta}
\def\te{\theta}
\def\pa{\partial}
\def\sup{\supset}
\def\sub{\subset}
\def\De{\Delta}
\def\I{\mathcal{I}}
\def\d{\delta}

\def\emp{\emptyset}
\def\ov{\overline}
\def\om{\omega}

\def\fa{\mathcal{F}}
\def\a{\alpha}
\def\be{\beta}
\def\R{\mathbb{R}}

\def\C{\mathbb{C}}

\def\{{\lbrace}
\def\}{\rbrace}

\def\la{\lambda}

\def\U{\mathcal{U}}

\def\L{\mathcal{L}}
\def\g{\gamma}
\def\P{\mathcal{P}}
\def\N{\mathbb{N}}
\def\Si{\Sigma}
\def\te{\theta}
\def\si{\sigma}
\def\G{\mathcal G}

\def\Z{\mathbb{Z}}
\def\O{\mathcal{O}}
\def\*{\star}

\def\wh{\widehat}

\def\n{\mathcal{N}}

\begin{document}

\title[Gauss map]{The Gauss map of a projective foliation}


\author[Claudia R. Alc\'antara]{C. R. Alc\'antara}

\author[D. Cerveau]{D. Cerveau}

\author[A. Lins Neto]{A. Lins Neto}

\address{$^{1}$ Departamento de Matem\'aticas, Universidad de Guanajuato, A.P 420, C.P. 36000, Guanajuato, Gto. M\'exico.} 
\email{claudia@cimat.mx}

\address{$^{2}$Universit\'e de Rennes I, IRMAR, Campus Beaulieu, 35042, Rennes Cedex, France.}
\email{dominique.cerveau@univ-rennes1.fr}

\address{$^{3}$ IMPA, Estrada Dona Castorina, 110, Horto, Rio de Janeiro, Brasil.}
\email{alcides@impa.br}

\thanks{The $1^{st}$ author was partially supported by Universidad de Guanajuato, Universit\'e de Rennes 1 and LaSol International Research Laboratory.}
\thanks{The $2^{st}$ author was supported by CNRS and Henri Lebesgue center, program ANR-11-LABX-0020-0.}
\thanks{The $3^{nd}$ author was partially supported by CNPq (Brazil) and Universit\'e de Rennes 1.}



\begin{abstract}
 In this paper, we study the Gauss map of a holomorphic codimension one foliation on the projective space $\p^n$, $n\ge 2$, mainly the case $n=3$.
Among other things, we will investigate the case where the Gauss map is birational.
\end{abstract}

\maketitle

\tableofcontents

\section{Introduction}

In a certain sense, this article gives us a generalization in the context of foliation theory of a classic concept in algebraic geometry, we refer to projective duality. The study of projective duality was started by Monge and his student Poncelet and later developed by Pl{\"u}cker, who was interested in the study of the dual curve of a plane curve (that is, the curve defined by the tangent lines). He proved that the degree of the dual curve $\breve\Gamma$ of a curve $\Gamma$ satisfies: $\deg{\breve\Gamma}=\deg{\Gamma}(\deg{\Gamma}-1)$, when the curve $\Gamma$ is smooth. 
For the singular case, the nature of the singularities of the curve must be taken into account: in this case, we have the famous Pl{\"u}cker's formulas. He also studied self-dual curves, which are those that are projectively equivalent to their dual. In the case of hypersurfaces with isolated singularities in a projective space $\p^n$, it is necessary to understand the works by Katz and Laumon to obtain significant results.
For the smooth case, we have an analogous formula to the previous one: $\deg{\breve{H}}=\deg{H}(\deg{H}-1)^{n-1}$, result that Laumon attributes to Katz. The case of hypersurfaces with singularities is more subtle, for which we refer to the work of Laumon (see \cite{laumon}). In this work he uses the invariants developed by Teissier in \cite{teissier}. These invariants generalize the classical Milnor number of singularities. While these topics still offer many directions, our main intention in this work is to obtain analogous results in the context of codimension one foliations.

As we shall see later, we can associate to a foliation
$\mathcal{F}$ of codimension 1 of the complex projective space $\p^n$, its Gaussian map. This is a rational map that sends to a non-singular point $m$, the hyperplane tangent to the leaf that passes through $m$. The space of codimension 1 foliations in $\p^n$ for $n>2$, is not completely known. Only spaces of degrees $0$, $1$ and $2$ in $\p^n$, for $n \geq 3$ have been fully described (see \cite{cl}). The case of degree 3 seems approachable after the works \cite{cerveau-lins-pisa} and \cite{rrj}. Given the limited understanding of these spaces, the Gauss map of a foliation could be a useful tool for the study of these spaces.


Motivated by this, in the present work we carry out a detailed study of this map. Concretely, the main objective of this article is to give some results on the Gauss map of codimension one foliations in $\p^n$, with special attention to the case where the map is birational. In addition, there are a large number of examples in which we make explicit calculations to clarify what happens in each case. 
This topic has been studied by some authors, such as Fassarella and Pereira in \cite{tj}, \cite{Fassarella2}, and \cite{fassarella}. 
\\

With the aim of offering a more insightful first reading, we describe in more detail the main results of the article.

We begin with \textbf{Theorem~\ref{pM}}, which asserts that the maximum topological degree (i.e., the cardinality of a generic fiber of the Gauss map) that a foliation of $\p^3$ of degree $d$ can attain is $d^2$, and when this occurs, the foliation admits a rational first integral
  of the form $\frac{F}{L^{d+1}}$, where $F$ has degree $d+1$ and $L$ has degree $1$. In \textbf{Theorem~\ref{tnp}}, we prove that if $F\colon\p^3\dasharrow\p^2$ is a rational map whose fibers are not plane curves, except for a finite number of them and there is a foliation $\G$ of $\p^2$ with the property that the foliation $\fa=F^*(\G)$ has a birational Gauss map and a contracted hypersurface that is not invariant, then $\fa$ has a rational first integral. \textbf{Theorem~\ref{p4}} provides, among other things, a characterization of foliations in $\p^3$ for which the Gauss map is birational. This characterization depends on the geometry of their singular set. Finally, \textbf{Theorem~\ref{pmon}} deals with monomial foliations in $\p^3$ and shows that if the Gauss map is dominant, then the foliations are of either a special logarithmic type or a specific pull-back from $\p^2$. In both cases, the Gauss map is birational.

Based on these results, we present several corollaries and numerous examples that illustrate, in some cases, that the hypotheses are sufficient but not necessary. Additionally, we study in detail the Gauss map of the foliations in the exceptional component constructed in
\cite{cl}. We also describe the algorithm we have used to construct the inverse of a Gauss map, whenever it exists.
\\

The article is structured as follows, in Section \ref{section2} we give the statements of structure, and we give a list of open problems that we find interesting. In Section \ref{ss2}, we can find the proofs of the results. Finally, in Section 4 we study the exceptional component of the space of foliations of codimension 1 in $\p^n$. 
\\

\noindent \textbf{Acknowledgement:} We would like to thank Serge Cantat and Julie Deserti for helpful discussions. We also thank Thiago Fassarella and Jorge V. Pereira for valuable comments and suggestions on this work. The first author is very grateful for the hospitality received at the University of Rennes. Finally, the authors are grateful to anonymous referees whose comments helped improve this article.

\section{Basic properties, Results, and Examples}\label{section2}

In order not to write many subscripts, we will adopt the following convention for the notation of the coordinates: If we are working in a projective space $\p^n$ with $n=2$, $n=3$ and $n=4$ we will use $(x,y,z,t,s)$. For general $n$, we will use $(z_0,...,z_{n})$. Except in Section 3.7 where, due to the nature of the proof of Theorem 4, it is more useful to work with $(z_1,z_2,z_3,z_4)$.
\\

Let $\fa$ be a codimension one holomorphic foliation on a $n$-dimensional complex manifold $X$, $n\ge2$, with singular set $Sing(\fa)$.
This foliation defines a codimension one distribution on $X\setminus Sing(\fa)$, denoted by $\de_\fa$: given $z\notin Sing(\fa)$ then $\de_\fa(z)$ is the tangent space at $z$ of the leaf of $\fa$ through $z$. We will always assume that $cod(Sing(\fa))\ge2$.

When $X$ is the complex projective space $\p^n$, then the distribution $\de_\fa$ gives origin to a meromorphic map $\G_\fa\colon \p^n\dasharrow\breve\p^n$, which we call the Gauss map, whose indetermination set is $Sing(\fa)$. It is defined as:
\[
z\in \p^n\setminus Sing(\fa)\mapsto \text{the hyperplane $\G_\fa(z)\in\breve\p^n$, tangent at $z$ to $\de_\fa(z)$.}
\]

In other words, if we fix an affine coordinate system $z=(z_1,...,z_n)\in\C^n\sub\p^n$ then $\fa$ is defined by some polynomial integrable 1-form, say $\om=\sum_{j=1}^nP_j(z)dz_j$. In this affine system, we have the following.
\[
Sing(\fa)\cap\C^n=\{z\in\C^n\,|\,P_j(z)=0\,,\,\forall\,1\le j\le n\}\,.
\]

In particular, if $z_o=(z_{o1},...,z_{on})\in \C^n\setminus Sing(\fa)$ then $\G_\fa(z_o)$ is the plane which, in this affine system, is given by $\ell(z)=0$, where
\[
\ell(z)=\sum_{j=1}^nP_j(z_o)(z_j-z_{oj})\,.
\]

The map $\G_\fa$ can be expressed also in homogeneous coordinates. This is done as follows: Let $\Pi\colon\C^{n+1}\setminus\{0\}\to\p^n$ be the canonical projection.
It is known that the foliation $\Pi^*(\fa)$ on $\C^{n+1}\setminus\{0\}$ can be extended to a foliation on $\C^{n+1}$, say $\wt\fa$. This foliation can be defined by an integrable polynomial 1-form, say $\wt\om=\sum_{j=0}^nQ_j(z)dz_j$, where the coefficients $Q_{j's}$ are all homogeneous polynomials of the same degree $d+1$, which is the degree of the foliation plus one, and such that $i_R\,\wt\om=0$, where $R=\sum_{j=0}^nz_j\frac{\pa}{\pa z_j}$, is the vector radial field in $\C^{n+1}$ and $i_R$ denotes the interior product by $R$. 

In homogeneous coordinates we have $Sing(\wt\fa)=\Pi^{-1}(Sing(\fa))\cup\{0\}=$
\[
=\{z=(z_0,...,z_n)\in \C^{n+1}\,|\,Q_0(z)=...=Q_n(z)=0\}
\]
and the map $\G_\fa$ is expressed as $\G_\fa[z]=[Q_0(z):...:Q_n(z)]\in\breve\p^n$.

Note also that $\Pi(Sing(\wt\fa)\setminus\{0\})$ is the indeterminacy set of $\G_\fa$. Now we summarize the above in the following definition, which we will use throughout the article.

\begin{defi}
    Let $\fa$ be a codimension one holomorphic foliation on $\p^n$, given by the 1-form $\omega=\sum_{j=0}^n Q_j(z)dz_j$. The Gauss map of $\fa$ is defined as the rational function:

    \begin{align*}
       \G_\fa: \p^n &\dashrightarrow  \breve\p^n\\
       z &\mapsto [Q_0(z):...:Q_n(z)].
    \end{align*}
\end{defi}

 As usual, we will say that $\G_\fa$ is dominant if it is generically of maximal rank.

\begin{rem}
\rm If $\G_\fa=[Q_0:...:Q_n]$ in homogeneous coordinates, then $\G_\fa$ is dominant if and only if $dQ_0\wedge dQ_1\wedge...\wedge dQ_n\not\equiv0$. 

Let $dQ_0\wedge dQ_1\wedge...\wedge dQ_n=J(z)\,dz_0\wedge...\wedge dz_n$. If $\G_\fa$ is dominant then
$Sing(\fa)\sub(J(z)=0)$. Let us prove this fact.

Let $R$ be the radial vector field on $\C^{n+1}$.
Since $Q_0,...,Q_n$ are homogeneous of the same degree, say $d$, by Euler's identity we have
\[
i_R[dQ_0\wedge dQ_1\wedge...\wedge dQ_n]=d\sum_{j=0}^n(-1)^jQ_j\,dQ_0\wedge...\wedge\wh{dQ_j}\wedge...\wedge dQ_n:=\Te.
\]
In the above formula the symbol $\wh{dQ_j}$ means the omission of $dQ_j$ in the product.
\vskip.1in
This implies that
\[
\Te=i_R[J(z)dz_0\wedge...\wedge dz_n]=J(z)\sum_{j=0}^n(-1)^jz_j\,dz_0\wedge...\wedge\wh{dz_j}\wedge...\wedge dz_n:=J(z)\Ga.
\]
On the other hand, if $z\in Sing(\fa)$ then $\Te(z)=0$, which implies $J(z)=0$, because $\Ga(z)\ne0$.\qed
\end{rem}

\begin{defi}
\rm Let $\fa$ be a foliation on $\p^n$ such that $\G_\fa$ is dominant.  
We say that $m\in\p^n\setminus Sing(\fa)$ is a point of Morse tangency for $\fa$ if the leaf $L_m$ of $\fa$ through $m$ has a Morse tangency with the hyperplane $\G_\fa(m)$. We say also that $L_m$ has a Morse tangency with $\G_\fa(m)$.

This means that there is a local coordinate system $z=(z_1,...,z_n)$ around $m$ such that $z(m)=0$, $\G_\fa(m)=(z_n=0)$ and the leaf $L_m$ has a local parametrization at $m$ of the form $z_n=z_1^2+...+z_{n-1}^2+h.o.t.$. 
\end{defi}

\begin{rem}\label{r12}
\rm Let $\fa$ be a foliation on $\p^n$. Then $\G_\fa$ is dominant if, and only if, the set 
\[
\U:=\{m\in\p^n\,|\,m\,\,\text{is a point of Morse tangency}\}
\]
is Zariski open and dense in $\p^n$. 

The proof is based on the fact that $m\in \U$ if, and only if, the derivative
$D\G_\fa(m)\colon T_m\p^n\to T_{\G_\fa(m)}\breve\p^n$
is an isomorphism. In fact, at a Morse tangency, the foliation is given by a form with linear part equivalent to: 
$$dx_n+x_1dx_1+...+x_{n-1}dx_{n-1},$$
and it is immediate to verify that $D\G_\fa(m)$ is an isomorphism.
Conversely, if $m$ is not a Morse tangency, the form defining the foliation at $m$ has linear part at $m$ equivalent to:
$dx_n+dq$, where $q$ is a quadratic form of rank less than $n-1$, so that $D\G_\fa(m)$ cannot be an isomorphism.
Finally, this implies that $D\G_\fa(\U)$ is Zariski open and dense if, and only if, $\U$ is also.

As a consequence, if $\fa=L^*(\fa')$, where $\fa'$ is a codimension one foliation on $\p^m$ with $m<n$, and $L\colon\p^n\dasharrow\p^m$ is a linear projection, then the Gauss map $\G_\fa$ cannot be dominant. In fact, if $z\in \p^n\setminus Sing(\fa)$ then the leaf $\L_z$ of $\fa$ through $z$ contains a straight line through $z$ such that $\fa$ is tangent to $\L_z$ along the line. In particular, the tangency with the tangent hyperplane cannot be of Morse type.
\end{rem}

In the case of foliations on $\p^3$ we have the following result. 

\begin{prop}\label{p3}
If the Gauss map of a foliation $\fa$ on $\p^3$ is not dominant then, either $\fa$ is a linear pull-back of a foliation on $\p^2$, or $\fa$ has a rational first integral $F$, which can be written in homogeneous coordinates as
\begin{equation}\label{eq}
F(x,y,z,t)=\frac{zP(x,y)+Q(x,y)}{tP(x,y)+R(x,y)}\,,
\end{equation}
where $P$, $Q$ and $R$ are homogeneous polynomials such that $dg(P)+1=dg(Q)=dg(R)$. 
\end{prop}

The proof of this proposition is a consequence of the results in \cite{cl}. We can also find it as theorem 4.5 in the Ph.D. Thesis by Fassarella (see \cite{Fassarella2}). In \cite{fassarella} we can find the classification of foliations on $\p^4$ with Gauss map not dominant. 

\begin{rem}
\rm We would like to observe that the map $F\colon\p^3\dasharrow\p^1$ of (\ref{eq}), in the statement of proposition \ref{p3} is birationally equivalent to the map $G(x',y',z',t')=z'/t'$.
In particular, the closure of the levels of $F$ in $\p^3$ are rational surfaces.

In fact, if $G\colon\p^3\dasharrow\p^3$ is defined as
\[
\Phi[x:y:z:t]=[x:y:zP(x,y)+Q(x,y):tP(x,y)+R(x,y)]:=[x':y':z':t']
\]
then $\Phi$ is birational and $G[x':y':z':t']=F\circ \Phi^{-1}[x':y':z':t']=z'/t'$.
\end{rem}

\begin{rem}
\rm Here we discuss the foliations of $\p^3$, with a fixed degree, inducing Gauss maps with the maximum number of points in the fiber.

Let $\mathbb{F}ol(d;n)$ be the space of foliations of codimension 1 of degree $d$ on $\p^n$ (we use the definition given in section 1 of \cite{cl}). Given $\fa\in\mathbb{F}ol(d;n)$ denote its Gauss map as $\G_\fa$. The topological degree of $\G_\fa$ is defined as $tdg(\G_\fa)=\#(\G_\fa^{-1}(m))$, where $m$ is a regular value of $\G_\fa$ (see \cite{mi}).
For instance, if $\G_\fa$ is not dominant then $tdg(\G_\fa)=0$, because a regular value of $\G_\fa$ is not in the image.

On the other hand, if $\G_\fa$ is dominant then $tdg(\G_\fa)<+\infty$, therefore we can define
\[
\n(d,n):=\{tdg(\G_\fa)\,|\,\,\fa\in \mathbb{F}ol(d;n)\}\,.
\]
and $\M(d,n)=max(\n(d,n))$. In Section 2.1 of Fassarella’s thesis (see \cite{Fassarella2}) several results on the topological degree are established; for example, it is shown that the map assigning to a foliation its topological degree is upper semicontinuous (Proposition 2.1). Proposition 2.6, on the other hand, contains a statement close to the one presented below. For the sake of completeness, we include the complet proof of this in \S \ref{sspM}.

\begin{thm}\label{pM}
$\M(d,3)=d^2$. 
Moreover, a foliation $\fa\in\mathbb{F}ol(d;3)$ such that $tdg(\G_\fa)=d^2$ must have a first integral of the form $F/L^{d+1}$, where $F$ is of degree $d+1$ and $L$ of degree one.
\end{thm}
\end{rem}

\begin{rem}
\rm A natural problem is to determine exactly the set $\n(d,n)$.
For instance, $\n(2,3)=\{0,1,2,3,4\}$. This is a consequence of the classification of degree two foliations of $\p^3$ (see \cite{cl} and example \ref{dg2}).
In example \ref{exdg3} we will see that $\n(3,3)=\{0,1,2,3,4,5,6,7,8,9\}$ and in remark \ref{ex5} that $1\in\n(d,3)$, for all $d\ge1$.

An example of a foliation on $\p^3$ of degree $d$ and with Gauss map with topological degree $d^2-d$ is the following:
let $d=p+q+r$ and let
\[
\om_d=fyz\left(\a \frac{df}{f}+ \be \frac{dy}{y} +\g \frac{dz}{z}\right)\,,
\]
where  $f=y^{q}z^{r}(xy^{p-1}+t^{p})+y^d+z^d$
and $\a d+\be+\g=0$. Using theorem \ref{p4} it is possible to prove that $tdg(\fa_{\om_d})=d^2-d$.
\\

Another interesting example is provided by the Fermat foliation of degree $d$ on $\p^n$, given by the 1-form:

$$\omega=\big(\sum_{i=1}^n z_i^{d+1}\big)dz_0-z_0z_1^ddz_1-...-z_0z_n^ddz_n,$$

\noindent where it can be seen that the topological degree of the Gauss map is
$d^{n-1}$.
\end{rem}


Let $m$ be a singular point of a foliation $\fa \in \mathbb{F}ol(d;n)$. We will say that $m$ is a Morse singularity for the foliation, if it admits a holomorphic first integral of the form
$f=Q+h.o.t.$, where $Q$ is a non-degenerate quadratic form.

\begin{rem}\label{rsing}
\rm If $\fa$ is a foliation in $\p^n$ with more than one Morse singularity then $\G_\fa$ cannot be birational.

For the proof, take $P$ a linear hyperplane through $m$ transversal to $\fa$ (i.e., the restriction of $Q$ to $P$ is non-degenerate), then the restriction of $\fa$ to $P$ has a Morse singularity at $m$. Furthermore, if $P'$ is a hyperplane sufficiently close to $P$, the restriction of $\fa$ to $P'$ has a Morse singularity $m'$ near $m$.

Now suppose that $\fa$ has two Morse singularities 
$a$ and $b$. Let $P$ be a hyperplane containing both, transversal to $\fa$ at these points. If $P'$ is a nearby hyperplane not containing $a$ and $b$, then $\fa$ has no singularities near $a$ and $b$, since Morse singularities are isolated. However, the restriction $\fa_{|P'}$ has two distinct Morse singularities. It follows that the Gauss map of $\fa$ takes the same value $P'$ at both points. Consequently, the Gauss map is not birational.

For instance, if $\fa$ has the first integral, given in affine coordinate by $f(x,y,z)=x^2+y^2+z^2-x^3$, then it has two Morse singularities in $\C^3$, namely, the points $(0,0,0)$ and $(2/3,0,0)$, so that $\G_\fa$ is not birational.

\vskip.1in
Note that if the foliation $\fa$ has less than two isolated singularities then $\G_\fa$ is not necessarily birational.
For instance, if $\fa$ has an isolated singularity with Milnor number $\mu\ge2$. A typical example is the foliation defined in affine coordinates $(x,y,z)\in\C^3$ by the 1-form $df$, where $f(x,y,z)=x^2+y^2+z^3$. 
\end{rem}

If $\G_\fa$ is birational, but not an isomorphism, then it necessarily contracts some hypersurface $V\sub\p^n$.

\begin{prop}\label{p1}
If $n\ge2$ then a birational map $\G\colon \p^n\dasharrow\p^n$, which is not an automorphism, always contracts a hypersurface of $\p^n$. That is, there exists a hypersurface $V\sub\p^n$ such that $\G(V)$ is an algebraic subset of $\p^n$ of dimension $<n-1$. 

More precisely, let $\hat{\G}\colon\C^{n+1}\to\C^{n+1}$ be the covering of $\G$, so that $\Pi\circ\hat{\G}=\G\circ\Pi$, where $\Pi\colon\C^{n+1}\setminus\{0\}\to\p^n$ is the canonical projection. If $\hat{\G}^{-1}\circ\hat\G=f.\,I$, where $f\in\C[z_0,...,z_n]$, then $\G$ contracts all irreducible components of $(f=0)$.

Conversely, if the hypersurface $(h=0)$ is contracted by $\G$ and $h$ is irreducible then $h\,|\,f$.
\end{prop}

The proof will be done in \S\,\ref{ss21}. See also \cite{cy}.

\begin{cor}
If the Gauss map $\G_\fa$ of a foliation $\fa$, of codimension one on $\p^n$ is birational then it always contracts some hypersurface $V\sub\p^n$.
Moreover, $V$ is contracted to a point of $\breve\p^n$ if, and only if, $V$ is an invariant hyperplane.
\end{cor}

Another interesting fact is the following:

\begin{prop}\label{p2}
Let $\fa$ be a foliation, of codimension one on $\p^n$ with birational Gauss map $\G$.
Then the Gauss map of the transformed foliation $\G_*(\fa)$ is also birational. Moreover, $\G_{\G_*(\fa)}=\G^{-1}$.
\end{prop}

The proof will be done in \S \ref{ss22}.

\vskip.1in

Next, we will see some examples for which the Gauss map is birational.

\begin{ex}\label{ex1}
{\it The case of foliations on $\p^2$.}
\rm Recall that the degree of a foliation on $\p^n$ is, by definition, the number of tangencies of a generic line $\ell\sub\p^n$ with the leaves of the foliation. For instance, if $\fa$ is a foliation of degree $d=0$ on $\p^2$ then $\fa$ is a radial foliation (a pencil of lines) and $\G_\fa(\p^2)\simeq\p^1\sub\breve\p^2$.

On the other hand, if $\fa$ is a foliation on $\p^2$ of degree $d\ge1$ then the Gauss map $\G_\fa$ is generically $d$ to one. Therefore, $\G_\fa\colon\p^2\dasharrow\breve\p^2$ is birational if, and only if $d=1$. Indeed $\n(d,2):=\{d\}$.

In homogeneous coordinates, a degree one foliation on $\p^2$ is defined by an action of $\C^2$ on $\C^3$ generated by the radial vector field $R=\sum_{j=0}^2z_j\frac{\pa}{\pa z_j}$ and a linear vector field $L$ of $\C^3$, such that $R\wedge L\ne0$. In the generic case, after a linear change of variables, we can suppose that $L=\sum_{j=0}^2\la_jz_j\frac{\pa}{\pa z_j}$, where $\la_i\ne\la_j$ if $i\ne j$.

In this case, $\fa$ is defined in homogeneous coordinates by the form
\[
\wt\om=\mu_0 z_1z_2dz_0+\mu_1 z_0z_2dz_1+\mu_2 z_0z_1dz_2\,,
\]
where $\mu_0=\la_2-\la_1$, $\mu_1=\la_0-\la_2$ and $\mu_2=\la_1-\la_0$.
In particular, the Gauss map has the expression:
\[
\G_\fa[z_0:z_1:z_2]=[\mu_0 z_1z_2:\mu_1 z_0z_2:\mu_2 z_0z_1]=[w_0:w_1:w_2]\,,
\]
which is birational. Its inverse has the expression
\[
\G_\fa^{-1}[w_0:w_1:w_2]=[\mu_0 w_1w_2:\mu_1 w_0w_2:\mu_2 w_0w_1]\,.
\]

The transformed foliation $\G_{\fa*}(\fa)$ is represented by the form
$\wt\om=\mu_0 w_1w_2dw_0+\mu_1 w_0w_2dw_1+\mu_2 w_0w_1dw_2$ and $\G_\fa$ contracts the three coordinate planes given by $(z_j=0)$, $0\le j\le2$, which are $\fa$ invariant.
\\

This is an example of what we will call \textbf{self-dual foliation}. As in the case of curves, this is a foliation such that the foliation given by the inverse of the birational Gauss map (if it exists), is projectively equivalent to the original.
\end{ex}

In the sequel, we will use the notation $\L(p_1,...,p_k;n)$ for the closure of the space of foliations on $\p^n$ defined in homogeneous coordinates by logarithmic 1-forms of the type
\[
\sum_{j=1}^k\la_j\frac{df_j}{f_j}\,,
\]
where $f_j\in\C\,[z_0,z_1,...,z_n]$ is homogeneous of degree $p_j$, $\lambda_j \in \mathbb{C}$, and $\sum_{j=1}^k\la_j p_j=0$.

\begin{ex}\label{ex2} 
\rm {\it The case of foliations of degree one on $\p^n$, $n\ge3$.} 
It is known (see chapter 2 of \cite{j}) that the space of foliations of degree one in $\p^n$, $n\ge3$, has two irreducible components:\\
1. The logarithmic component $\L(1,1,1;n)$ whose generic elements are, as we said before, defined in homogeneous coordinates by 1-forms linearly equivalent to the following:
\[
\om=xyz\,\left(\la_1\frac{dx}{x}+\la_2\frac{dy}{y}+\la_3\frac{dz}{z}\right)\,,
\]
where $\la_j\in\C^*$, $j=1,2,3$, and $\la_1+\la_2+\la_3=0$. In this case, the foliation $\fa_\om$ is the linear pull-back of the foliation on $\p^2$ defined by the same form. The map $\G_{\fa_\om}$ is degenerated and its generic rank is two.
\\

2. The component $\L(2,1;n)$, whose elements have a rational first integral of the form $\frac{q}{\ell^2}$, where $q$ is irreducible of degree two and $\ell$ has degree one.

In the generic case, the Gauss map of the associated foliation is birational and not a linear isomorphism from $\p^n$ to $\breve\p^n$. In fact, in this case, we can assume that in some homogeneous coordinates $z=(z_0,z_1,...,z_n)$ we have $q(z)=\sum_{j=1}^nz_j^2$ and $\ell(z)=z_0$.
The foliation is defined by the form
\[
\om=\frac{1}{2}\,\ell.\,q\left(\frac{dq}{q}-2\frac{dz_0}{z_0}\right)=-q\,dz_0+z_0\sum_{j=1}^nz_jdz_j\,,
\]
so that in homogeneous coordinates we have
\[
\G_\fa(z)=\left(-\sum_{j=1}^nz_j^2,z_0z_1,...,z_0z_n\right)=(y_0,y_1,...,y_n)=y\,.
\]
Its inverse is
\[
\G_\fa^{-1}(y)=\left(-\sum_{j=1}^ny_j^2,y_0y_1,...,y_0y_n\right)\,,
\]
thus
\[
\G_\fa^{-1}\circ\G_\fa(z)=-z_0\sum_{j=1}^nz_j^2\,(z_0,z_1,...,z_n)=-\ell\,q.\,z
\]
As a consequence $\G_\fa$ contracts the hypersurfaces $(\ell=0)$ and $(q=0)$, which are leaves of $\fa$.
The hyperplane $(z_0=0)$ is contracted to the point $[1:0:...:0]$ and the quadric $(q=0)$ to the quadric of codimension two $(y_0=0)\cap(q(y)=0)$. Note that this case, in dimension three, is one of the homaloids, that we will study in example \ref{hom}.   

In the non-generic case, it can be verified that the Gauss map is not dominant.
\end{ex}

\begin{ex}\label{ex3}   
\rm A generalization of the generic case of example \ref{ex1} is the foliation $\fa$ on $\p^n$ defined in homogeneous coordinates by the 1-form
\begin{equation}\label{eqlog}
\om=z_0z_1...z_n\sum_{j=0}^n\mu_j\frac{dz_j}{z_j}\,,
\end{equation}
where $\mu_j\in\C^*$, $0\le j\le n$, and $\sum_{j=0}^n\mu_j=0$. This last condition is equivalent to $i_R\,\om=0$.

Note that this foliation is in the irreducible component $\L(1,...,1;n)$ ($n+1$ ones) of the space of codimension one foliations on $\p^n$ of degree $n-1$. 
The Gauss map in this case is
\[
\G_\fa[z_0:z_1:...:z_n]=[\mu_0z_1...z_n:\mu_1z_0z_2...z_n:...:\mu_nz_0z_1...z_{n-1}]\,,
\]
which is birational and whose inverse has a similar expression; this is a self-dual foliation. The map $\G_\fa$ contracts the union of the hyperplanes $(z_j=0)$, $0\le j\le n$.
\end{ex}

\begin{rem}\label{ex5}
\rm  Here we will see that for any degree $d\ge1$ there are foliations on $\p^3$ whose associated Gauss map is birational.

Let $\fa_k$ be the foliation on $\p^3$ with the first integral, which in homogeneous coordinates is written as
\[
F_k(x,y,z,t)=\frac{g}{y^{k+2}}=\frac{(xz+yt)y^k+x^{k+2}}{y^{k+2}}.
\]
Then the foliation is defined by the 1-form
\[
\om=yg\left[\frac{dg}{g}-(k+2)\frac{dy}{y}\right]=
\]
\[
=y[(k+2)x^{k+1}+y^kz]dx-[(k+2)x^{k+2}-2xy^kz-y^{k+1}t]dy+xy^{k+1}\,dz+y^{k+2}\,dt,
\]
and the Gauss map is\,\, $\G_{\fa_k}[x:y:z:t]=[A:B:C:D]$
\[
=[y[(k+2)x^{k+1}+y^kz]:-[(k+2)x^{k+2}-2xy^kz-y^{k+1}t]:xy^{k+1}:y^{k+2}].\,
\]

In this example we have that the singular set is $(x=y=0):=\ell$, the hyperplane $\Si=(y=0)$ is $\fa$-invariant and this is contracted to a point by $\G_{\fa_k}$. This Gauss map is birational with inverse $\G^{-1}_{\fa_k}[A:B:C:D]$
\[
=[CD^{k+1}:D^{k+2}:AD^{k+1}-(k+2)C^{k+1}D:BD^{k+1}-2ACD^{k}+3(k+2)C^{k+2}].\,
\]

\vskip.1in 
We would like to remark that if $k=1$ then $\fa_1\in \L(3,1;3)\sub\fol(2;3)$. For a generic member $\fa\in \L(3,1;3)$ it can be shown that  $tdg(\G_\fa)=4$, so that $\G_\fa$ is not birational (see example \ref{dg2}).
 

If $k\ge2$ then $\fa_k\in \L(k+2,1;3)\sub\fol(k+1;3)$. For a generic member $\fa\in\L(k+1,3;3)$ the singular set is a plane curve of degree $k+2$, but in the case of $\fa_k$, the singular set degenerates in a straight line of multiplicity $k+2$.
\end{rem}

\begin{ex}\label{exAB}
Magic example.\rm\,\, Here is an example for which the associated foliation has no special transversal structure and contracts hypersurfaces that are not invariant for the foliation. Let $\om=A(u,v)du+B(u,v)dv$, where $A$ and $B$ are non-zero polynomials. We consider the pull-back map $\Phi(x,y,z,t)=(x/t,y/z)$. Note that
\[
\Phi^*(\om)=A(x/t,y/z)d\left(\frac{x}{t}\right)+B(x/t,y/z)d\left(\frac{y}{z}\right):=\te.
\]
Then we have
\[
\te=A(x/t,y/z)\,\frac{t\,dx-x\,dt}{t^2}+B(x/t,y/z)\,\frac{z\,dy-y\,dz}{z^2}
\]
and so
\[
z^2t^2\,\te=z^2t\,A(x/t,y/z)dx+zt^2\,B(x/t,y/z)dy-yt^2\,B(x/t,y/z)dz-xz^2\,A(x/t,y/z)dt
\]

In homogeneous coordinates the Gauss map is $\G[x:y:z:t]=$
\[
=[z^2t\,A(x/t,y/z):zt^2\,B(x/t,y/z):-yt^2\,B(x/t,y/z):-xz^2\,A(x/t,y/z)]\,.
\]

In particular, the hypersurfaces $(A(x/t,y/z)=0)$ and $(B(x/t,y/z)=0)$ are contracted by $\G$.
These hypersurfaces in general are not invariant for the foliation defined by $\te$.

The Gauss map $\G$ associated to the foliation $\Phi^*(\fa_\om)$ is birational because in the affine coordinate system $t=1$ we have
\[
\G(x,y,z)=\left(-\frac{1}{x},-\frac{\,B(x,y/z)}{xz\,A(x,y/z)},\frac{y\,B(x,y/z)}{xz^2\,A(x,y/z)}\right)
\]
We can solve $\G(x,y,z)=(a,b,c)$ as
\[
x=-1/a\,,\,y/z=-c/b\,,\,z=-\frac{B(-1/a,-c/b)}{b\,x\,A(-1/a,-c/b)}\,,\,y=-\frac{cz^2\,A(-1/a,-c/b)}{aB(-1/a,-c/b)}\,.
\]

We note that, in general, the original foliation in $\p^2$, has no invariant algebraic curves as in the Jouanolou foliation \cite{ls}  (see chapter 4 of \cite{j}). Sometimes can even have some dense leaves (see \cite{fl} and \cite{lr}).

With this method, we can obtain foliations on $\p^3$ with birational Gauss map of any desired degree $\ge2$, without rational first integral. 

To prove this statement, it suffices to consider $\om=f(u)du+dv$, where $f(u)$ is a polynomial of degree $s$. Applying the pull-back by 
$\Phi$ in the affine chart $t=1$, we obtain $z^2f(x)dx+zdy-ydz$, which defines a foliation of degree $s+2$. 


{\bf Note:} We have called this example "magic example" because we don't see a geometric reason why $\fa=\Phi^*(\G)$ always has a birational Gauss map.
\end{ex} 

In example \ref{exAB} all fibers of the pull-back map $[x:y:z:t]\dasharrow (x/t,y/z)$ are straight lines. 
A natural question is: what happens if the fibers are not plane curves?

In this case, we have the following:

\begin{thm}\label{tnp}
Let $F\colon\p^3\dasharrow\p^2$ be a rational map whose fibers are not plane curves, except for a finite number of them.
Suppose that there is a foliation $\G$ of $\p^2$ with the property that the foliation $\fa=F^*(\G)$ has Gauss map birational and a contracted hypersurface which is not invariant.
Then $\fa$ has a rational first integral.
\end{thm}

We will prove Theorem \ref{tnp} in \S\,\ref{ss?}.
Let us see an example illustrating the importance of the hypotheses of the previous Theorem.

\begin{ex}
\rm 
Consider the rational map $F\colon\p^3\dasharrow\p^2$, given in affine coordinates by $(x,y,z) \mapsto (y+x^2,z)$. We take the foliation $\G$ of $\p^2$ given locally by the 1-form $(u-v) du+\lambda udv$, where $\lambda \in \C^*$. We can see that $\fa=F^*(\G)$ is defined by $\omega=Adx+Bdy+Cdz+Ddt,$ where

\begin{align*}
A&=2xt(yt+x^2-zt),\\
B&=t^2(yt+x^2-zt),\\
C&=\lambda t^2(yt+x^2),\\
D&=(yt+2x^2)(zt-yt-x^2)-\lambda tz(ty+x^2),
\end{align*}

\noindent and its Gauss map is birational. Its  inverse is:

\begin{align*}
\G_\fa^{-1}[A:B:C:D]&=[2\lambda ABC((\lambda-1)B-C):\lambda C((2-\lambda)A^2B+4B^2D+A^2C):\\
&B(\lambda A^2+4\lambda BD)(C-\lambda B):4\lambda B^2C((\lambda-1)B-C)].    
\end{align*}

The facts to be noted in this example are the following: The fibers of $F$ are conic plane curves, the hypersurface given by 
$\{yt+x^2-zt=0\}$ is not invariant by $\fa$, however, it is contracted by the Gauss map $\G_\fa$. Finally, we can also note that for $\lambda$ generic, the foliation $\fa$ has no rational first integral, because the 1-form $(u-v) du+\lambda udv$ does not have it.

It is also interesting to relate this example to the magic example (see example \ref{exAB}) where all fibers of the rational map are straight lines and the pull-back of every foliation of $\p^2$ has birational Gauss map.
\end{ex}
\vskip.1in

Before stating the next results, let us recall some facts about a codimension one foliation $\fa$ on $\p^n$.

The singular set of $\fa$ always has irreducible components of codimension two (see for example \cite{ln}).
We will denote as $Sing_2(\fa)$ the union of the irreducible components of codimension two of $Sing(\fa)$.
\\
If $\fa$ is a foliation on $\p^2$, the Milnor number $\mu(\fa,p)$, of $\fa$ at a point $p$ is defined as the dimension of the local ring at $p$ modulo the ideal generated by the local vector field defining the foliation near $p$ (see \cite{CLNS}).
\\

 Let $\fa \in \mathbb{F}ol(d;n)$, where $n\ge3$, and let $V\sub Sing_2(\fa)$ be an irreducible component of codimension two. Fix a smooth point $p\in V$ and a transverse 2-dimensional disk $\Si$ through $p$ such that $\Si$ is transverse to $\fa$ outside $p$. Since the restriction $\fa|_\Si$, is a foliation on $\Si$ with an isolated singularity at $p$ we can define the Milnor number $\mu(\fa|_\Si,p)$. It is called the multiplicity of $\fa$ at $(\Si,p)$.

\begin{defi} 
We define the transversal multiplicity of $V$ as follows. Firstly we called:

\[
\mu(\fa,V,p)=min\{\mu(\fa|_\Si,p)\,|\,\Si\,\text{is a transverse section through $p\in V\}$}\,.
\]

\noindent multiplicity of $\fa$ at $p\in V$. It can be proved that the function $p\in V\mapsto \mu(\fa,V,p)\in\N$ is constant in a Zariski open and dense subset of $V$.
This constant will be denoted as $\mu(\fa,V)$, and it is called the transverse multiplicity of $\fa$ at $V$.
\end{defi}

\begin{rem}
\rm If $V$ has a Zariski open and dense set, say $V'\sub V$, of locally product points (see \cite{ln1}) then the number $\mu(\fa,V',p)$ is constant. A typical case is when all points of $V'$ are of Kupka type (see \cite{ln1}).

For instance, when $\fa$ is defined by a generic logarithmic form that is the case (see proposition \ref{p5}). 
\end{rem}

\begin{thm}\label{p4}
Let $\fa$ be a foliation on $\p^3$ of degree $d$. Let $Sing_2(\fa)=\bigcup_{j=1}^kV_j$, where $V_1,...,V_k$ are the irreducible components of dimension one of $Sing(\fa)$. Then:
\begin{itemize}
\item[(a)] $\sum_{j=1}^k\mu(\fa,V_j).deg(V_j)\le d^2+d+1$. The equality is attained when $\fa$ is a pull-back of a foliation on $\p^2$ by a linear map $\Psi\colon\p^3\dasharrow\p^2$ (see \cite{cl} for a complete description of these foliations).
\item[(b)] Let $m(\fa)=d^2+d+1-\sum_{j=1}^k\mu(\fa,V_j).deg(V_j)$, 
then $\G_\fa\colon\p^3\dasharrow\breve\p^3$ has topological degree $m(\fa)$.
In particular, $\G_\fa$ is dominant if and only if, $m(\fa)> 0$; and $\G_\fa$ is birational if and only if, $m(\fa)=1$.
\end{itemize}
\end{thm}

\begin{cor}\label{c}
Let $\fa_o$ be a foliation of degree $d$ on $\p^3$. Assume that there is a 2-plane $\p^2\simeq\Si\sub\p^3$ such that the restricted foliation $\fa_o|_\Si$ has $d^2+d$ singularities with generic linear part, in the sense that the determinant and the trace are non zero. Suppose further that the Gauss map $\G_{\fa_o}$ is dominant. Let $Z$ be the irreducible component of $\fol(3,d)$ containing $\fa_o$. Then:
\begin{itemize}
\item[1.] $\G_{\fa_o}$ is birational.
\item[2.] If $\fa\in Z$ is a generic foliation in the same component then its Gauss map is birational.
\end{itemize}
\end{cor}

The following three corollaries of the previous Theorem are valid in all dimensions, and give us sufficient conditions on the singular set of the foliation, to have birational Gauss map. The first and third are similar except for a small simplification of the formula described. The second one is for restrictions of foliations.

\begin{cor}\label{c12}
Let $\fa$ be a foliation on $\p^n$ of degree $d$, $n\ge4$.
Let $V_1,...,V_k$ be the irreducible components of $Sing_2(\fa)$. If $\sum_{j=1}^k \mu(\fa,V_j)\,deg(V_j)=d^2+d$, then $\G_\fa$ is birational.
\end{cor}

\begin{cor}\label{new}
Let $\fa$ be a foliation on $\p^n$ of degree $d$, $n\ge 3$.
Let $V_1,...,V_k$ be the irreducible components of $Sing_2(\fa)$.
Assume that $\sum_{j=1}^k \mu(\fa,V_j)\,deg(V_j)=d^2+d$. Let $3\leq m \leq n$, and suppose that $\p^m$ is linearly embedded in $\p^n$, then the Gauss map of the restriction $\fa_{|\p^m}$ is birational.
\end{cor}

\begin{cor}\label{c14}
Let $\fa$ be a foliation on $\p^n$ of degree $d$, $n\ge4$.
Let $V_1,...,V_k$ be the irreducible components of $Sing_2(\fa)$. If $\sum_{j=1}^k deg(V_j)=d^2+d$, then $\G_\fa$ is birational.
\end{cor}

\begin{rem}\label{rem15}
\rm We would like to note that the hypothesis of corollary \ref{c12}\, is sufficient for the Gauss map of a foliation on $\p^n$ to be birational, but is not necessary when $n\ge4$. 
For instance, in the case of example \ref{ex3}, where $\om=z_0z_1...z_n\sum_{j=0}^n\mu_j\frac{dz_j}{z_j}$, we have:\\
1. $Sing(\fa_\om)=\bigcup_{i<j}V_{ij}$, where $V_{ij}=(z_i=z_j=0)$.\\
2. $\mu(\fa_\om,V_{ij})=1$ and $deg(V_{ij})=1$, $\forall i<j$.\\
This implies $\sum_{i<j}\mu(\fa_\om,V_{ij})\,deg(V_{ij})=\frac{n(n+1)}{2}=:S$.\\
3. Since $deg(\fa_\om)=n-1:=d$ we get
\[
(*)\hskip.3in d^2+d+1-S=\frac{(n-1)(n-2)}{2}:=m
\]
As the reader can check, we have $m=1$ if, and only if, $n=3$.
Therefore if $n\ge4$ the condition in corollary \ref{c12} is not necessary.
\end{rem}

Let us see some applications of theorem \ref{p4}.

\begin{ex}\label{hom}
\rm Let $F\in\C[z_1,...,z_n]$ be a homogeneous polynomial. We say that $F$ is a {\it homaloid} if 
the map $dF:=(F_{z_1},...,F_{z_n})\colon\C^n\to\C^n$ is birational.

The problem of classification of homaloids was studied originally by Dolgachev in \cite{do}.
Essentially he proved that the reduced homaloids in $\C^3$ are the following:\\ 
(1) A smooth conic: $F(x,y,z)=x^2+y^2+z^2$. The foliation, given by the level sets of $F$, has degree one.\\
(2) The union of three lines in general position: $F(x,y,z)=xyz$. The foliation in this case has degree two (see example \ref{dg2}).\\
(3) The union of a smooth conic and a line tangent to it: $F(x,y,z)=y(x^2+yz)$. The foliation has degree two and is in the component $\L(2,1,1;3)$, however, it is not generic in the component (see \cite{cl}).

A. Dimca and S. Papadima also studied this problem in \cite{dp}. They have expressed
the topological degree of the Gauss map in terms of the variety $(F=0)$.

T. Fassarella and J. Vitório Pereira have also considered this problem (see \cite{tj}), but from the point of view of the theory of foliations.

All these examples can be verified using Theorem \ref{p4}. 
Let us verify case (3). In homogeneous coordinates, it has a rational first integral of the form $F(x,y,z,t)=\frac{y(x^2+yz)}{t^3}$, whose foliation $\fa$ is defined by the form 
\[
\om=2xyt\,dx+t(x^2+2yz)\,dy+y^2t\,dz-3y(x^2+yz)\,dt,
\]
so that $Sing_2(\fa)=V_1\cup V_2\cup V_3$, where $V_1=(x=y=0)$, $V_2=(y=t=0)$ and $V_3=(t=x^2+yz=0)$.
It can be verifed that $\mu(\fa_\om,V_1)=3$, $\mu(\fa_\om,V_2)=1$ and $\mu(\fa_\om,V_3)=1$.
For instance, $\Si=(z=t=1)$ is a transversal section to $V_1$ and
\[
\om|_\Si=2xy\,dx+(x^2+2y)\,dy\,\implies\,\mu(\fa_\om,V_1)=3\,.
\]
Since $deg(V_1)=deg(V_2)=1$ and $deg(V_3)=2$, we get $\sum_j\mu(\fa_\om,V_j)deg(V_j)=6$. On the other hand, the degree of the foliation is two, and so $d^2+d+1-6=1$, so that $\G_\fa$ is birational.
\end{ex}

\begin{rem}
\rm A particular case of example \ref{exAB} is $\om=u^pdu+v^qdv$, where $p,q\in\N$. In this case, we get
\[
\te=\frac{x^p(ydx-xdy)}{y^{p+2}}+\frac{z^q(tdz-zdt)}{t^{q+2}}=
\]
\[
=\frac{1}{y^{p+2}t^{q+2}}\left(x^pyt^{q+2}dx-x^{p+1}t^{q+2}dy+y^{p+2}z^qtdz-y^{p+2}z^{q+1}dt\right)
\]
whose Gauss map is 
\[
\G[x:y:z:t]=[x^pyt^{q+2}:-x^{p+1}t^{q+2}:y^{p+2}z^qt:-y^{p+2}z^{q+1}]
\]
with inverse $\G^{-1}[a:b:c:d]=$
\[
=[(-1)^{q+1}b^{p+1}\,c^{q+2}:(-1)^qa\,b^p\,c^{q+2},(-1)^{p+1}a^{p+2}\,d^{q+1}:(-1)^pa^{p+2}\,c\,d^q].
\]

We say that a foliation on $\p^n$ is {\it monomial} if it is defined in homogeneous coordinates by an integrable polynomial 1-form of the type
\begin{equation}\label{eqmon}
\om=\sum_{j=0}^nM_j(z)dz_j\,\,,
\end{equation}
where the $M_{j's}$ are monomials of the same degree; $M_j(z)=\la_j\,z^{\si_j}$. As usually, we use the notations $\si=(\si(0),...,\si(n))$ 
and $z^\si=z_0^{\si(0)}.z_1^{\si(1)}...z_n^{\si(n)}$. 
We will assume that in (\ref{eqmon}) the monomials are not scalar multiples of one particular monomial:
$[M_0:M_1:...:M_n]\ne [\la_0:\la_1:...:\la_n]\in\p^n$.
A typical example in $\p^3$ is the one with $\om=u^pdu+v^qdv$ with $u=x/y$ and $v=z/t$.

Here we will prove the following:

\begin{thm}\label{pmon}
Let $\om=\sum_{i=1}^4A_i(z)dz_i$ be a monomial integrable 1-form on $\C^4$, defining a monomial foliation $\fa_\om$ on $\p^3$ of degree $\ge2$. If the Gauss map $\G_\om$ of $\fa_\om$ is dominant then $\fa_\om$ belongs to $\L(1,1,1,1;3)$; or it is a pull-back of a form on $\C^2$ like in example \ref{exAB}, for degree $\ge3$. In particular, $\G_\om\colon\p^3\dasharrow\p^3$ is birational.
\end{thm}

The proof of Theorem \ref{pmon} will be done at \S \ref{sspmon}.
\end{rem}

\begin{ex}\label{ex4}
\rm Here is another application of theorem \ref{p4}. The exceptional foliation of degree two on $\p^3$,  $\E\in\fol(2;3)$, is defined in homogeneous coordinates by the form 
\[  
\om=t(2y^2-3xz)dx+t(3tz-xy)dy+t(x^2-2yt)dz+(2x^2z-yzt-xy^2)dt
\]

This example is studied in detail in \cite{cl}.

In homogeneous coordinates, it has a rational first integral of the form $F=f^2/g^3$, where $deg(f)=3$ and $deg(g)=2$. It is called exceptional because the logarithmic form $\frac{dF}{F}=2\frac{df}{f}-3\frac{dg}{g}$ has $t$ as divisor of zeroes.
We note that $f(x,y,z,t)=zt^2-xyt+x^3/3$, $g(x,y,z,t)=yt-x^2/2$ (see \cite{cl}),
$fg\frac{dF}{F}=t\om$ and $cod(Sing(\om))=2$, so that $\om$ defines $\E$ in homogeneous coordinates.
We would like to observe that it was proved in \cite{cl} that the irreducible component of $\fol(2,3)$ that contains the exceptional foliation is the closure of the orbit of $\E$ by the natural action of $Aut(\p^3)$. From now on this component will be denoted as $\E(2,3;3)$.

In \cite{cl} it is also proven that $Sing_2(\E)$ has three irreducible components:\\
1. The closure in $\p^3$ of the twisted cubic $V_3=\{(s,s^2/2,s^3/6)\,|\,s\in\C\}\sub\C^3=\p^3\setminus(t=0)$.\\
2. The straight line $V_1=(x=t=0)$.\\
3. A conic $V_2=(y^2-2xz=t=0)$.

It is verified also that $\mu(\fa,V_j)=1$, $j=1,2,3$.
Therefore, in this case we have $m(\E)=7-(3+2+1)=1$ and $\G_\E$ is birational.
This example will be studied in detail in \S\,\ref{ss23'} and in \S\,\ref{ss}. For example, the inverse
$\G_\E^{-1}$ will be done in \S\,\ref{ss23'}.
\end{ex}

\begin{ex}\label{e34} 
\rm Another exceptional example, but now in $\p^4$, has also a rational first integral of the form $F=\frac{f^3}{g^4}$, where $deg(f)=4$ and $deg(g)=3$. Let $\E'$ be the foliation on $\p^4$ having the following rational first integral in homogeneous coordinates $(s,x,y,z,t)\in\C^5$:
\[
F(s,x,y,z,t)=\frac{\big(st^3-(2xz+y^2)t^2+2yz^2t-\frac{1}{2}z^4\big)^3}{\big(xt^2-yzt-\frac{1}{3}z^3\big)^4}:=\frac{f^3}{g^4}
\]

The logarithmic form $\frac{dF}{F}=3\frac{df}{f}-4\frac{dg}{g}$ has $t^2$ as zero divisor. 
In fact, a direct computation shows that $fg\frac{dF}{F}=t^2\om$, where the components of $\om$ have degree $4$, so that $deg(\fa)=3$. 
This example will be studied in detail in \S\,\ref{ss33}. In \S\,\ref{ss33} we will see that:
\begin{itemize}
\item[(a).] $Sing_2(\E')$ has three irreducible components, which we denote as $V_1$, $V_2$ and $V_3$.
\item[(b).] $\mu(\fa,V_i)=1$, $1\le i\le 3$.
\item[(c).] $deg(V_1)=1$, $deg(V_2)=3$ and $deg(V_3)=8$.
\end{itemize}

In particular, we get $\sum_{j=1}^3\mu(\E',V_i)\,deg(V_i)=12$, so that the Gauss map of $\E'$ is birational, by theorem \ref{p4}.

We would like to observe that this example was originally introduced in \cite{jc}. It was proved in \cite{cp} that the irreducible component of $\fol(3;4)$ that contains $\E'$ is the closure of its orbit by the natural action of $Aut(\p^4)$. From now on this component will be denoted as $\E(3,4;4)$.
\vskip.1in
We define $\E(3,4;n)$, $n\ge4$, as the set
\[
\E(3,4;n)=\{L^*(\E(3,4;4))\,|\,L\,\,\text{is a linear map}\}\,.
\] 
It can be proved that $\E(3,4;n)$ is an irreducible component of $\fol(4;n)$ for all $n$ (see \cite{cp}). However, if $\E\in\E(3,4;n)$ and $n>4$ then $\G_\E$ is not dominant.
\end{ex}

\begin{rem}\label{exc}
\rm In \S\,\ref{ssexc} we will generalize these examples of exceptional foliations for foliations of degree $n-1$ on $\p^n$. As in the case of examples \ref{ex4} and \ref{e34}, these examples will have a rational first integral of the form $f^{n-1}/g^n$, where $deg(f)=n$ and $deg(g)=n-1$.
\end{rem}

\begin{ex}\label{exn}
According to \cite{cl}, the space of foliations of degree two on $\p^n$ where $n\ge3$, has the following six irreducible components:
\begin{itemize}
\item[(a).] $PB(2\,;n)$; the component containing pull-back of foliations of degree two on $\p^2$ by a linear map $L\colon \p^n\dasharrow\p^2$.
\item[(b).] $\E(2,3\,;n)$; the component containing the exceptional foliations on $\p^n$.
\item[(c).] $\L(1,1,1,1\,;n)$; component of foliations of degree two, defined by logarithmic forms of the type $\sum_{j=1}^4\la_j\frac{d\ell_j}{\ell_j}$, where $deg(\ell_j)=1$, $1\le j\le4$.
\item[(d).] $\L(1,1,2\,;n)$; component of foliations of degree two, defined by logarithmic forms of the type $\la_1\frac{d\ell_1}{\ell_1}+\la_2\frac{d\ell_2}{\ell_2}+\la_3\frac{dq}{q}$, where $deg(\ell_1)=deg(\ell_2)=1$ and $deg(q)=2$.
\item[(e).] $\L(1,3\,;n)$;  component of foliations of degree two, with a first integral of the form $f/\ell^3$, where $deg(\ell)=1$ and $deg(f)=3$.
\item[(f).] $\L(2,2\,;n)$; component of foliations of degree two, with a first integral of the form $q_1/q_2$, where $deg(q_1)=deg(q_2)=2$.
\end{itemize}

Let $n \geq 4$, we will see that the Gauss maps of the generic foliations of degree two on $\p^n$ are not birational. Examples (a), (b), and (c), cannot be dominant on $\p^n$.
For instance, an exceptional foliation on $\p^n$, is a linear pull-back of the exceptional foliation $\E$ on $\p^3$ (see \cite{cl}), and so its Gauss map cannot be dominant.

Generic foliations in components (d), (e), and (f), have dominant Gauss maps, but not birational.
For instance, in case (e) we can assume that $f$ is a generic cubic polynomial. In homogeneous coordinates $(s,t,x,y,z)$ we can suppose that $\ell=s$, so that in the affine coordinate system $s=1$ the foliation $\fa\in \L(1,3\,;n)$ has the first integral $f(1,t,x,y,z)$. The hyperplane $(t=c)$ is tangent to $\fa$ at the points for which $f_x(1,c,x,y,z)=f_y(1,c,x,y,z)=f_z(1,c,x,y,z)=0$. If $f$ and $c$ are generic, this set contains $8$ points, whose image under the Gauss map is the hyperplane $(t=c)$, this implies that the map cannot be birational.

However, there are degenerated foliations in some of these components whose Gauss map is birational.
For instance, in $\L(1,1,2;4)$ we have the foliation $\fa$ with first integral
\[
f(s,t,x,y,z)=\frac{x^2+y^2+z^2}{st}:=\frac{q}{st}
\]
which is defined in homogeneous coordinates by the form
\[
\om=st\,q\left(\frac{dq}{q}-\frac{ds}{s}-\frac{dt}{t}\right)=st\,dq-tq\,ds-sq\,dt\,.
\]

The Gauss map in the affine coordinate system $t=1$ is
\[
\G_\fa(x,y,z,s)=\left(-\frac{2x}{q},-\frac{2y}{q},-\frac{2z}{q},\frac{1}{s}\right)=(a,b,c,e)
\]
and it has as inverse
\[
\G_\fa^{-1}(a,b,c,e)=\left(-\frac{2a}{q(a,b,c)},-\frac{2b}{q(a,b,c)},-\frac{2c}{q(a,b,c)},\frac{1}{e},\right)
\]

\noindent then, the foliation $\fa$ is self-dual.
\end{ex}

\begin{ex}\label{dg2}
 In this example, we use the notation for the components of $\mathbb{F}ol(2;3)$, introduced in the previous one. Using Theorem \ref{p4} it is not difficult to see that there are two of these components for which the generic elements have birational Gauss map: the exceptional component of example \ref{ex4} and the logarithmic component $\L(1,1,1,1;3)$ of example \ref{ex3}.

Using Theorem \ref{p4} we obtain that:\\
1. A generic foliation in $\fa\in\L(2,1,1;3)$ has $tdg(\G_\fa)=2$.\\
2. A generic foliation in $\fa\in\L(3,1;3)$ has $tdg(\G_\fa)=4$.\\
3. A generic foliation in $\fa\in\L(2,2;3)$ has $tdg(\G_\fa)=3$.

This calculation can also be found in example 2.5 of \cite{Fassarella2}. We observe that 
$\n(2,3)=\{0,1,2,3,4\}$. 

However, there are degree two foliations on $\p^3$, in the boundary of the components, for which the Gauss map is birational.
An example is the homaloid (3) of example \ref{hom}.
This foliation, in homogeneous coordinates, has a first integral of the form
\[
F(x,y,z,t)=\frac{y(x^2+yz)}{t^3}:=\frac{f(x,y,z)}{t^3}
\]
and belongs to $\L(1,3;3)\cap \L(1,1,2;3)$.
The foliation is defined in homogeneous coordinates by the form
\[
\om=t\,f\left(\frac{df}{f}-3\frac{dt}{t}\right)=2xyt\,dx+t(x^2+2yz)dy+y^2t\,dz-3y(x^2+yz)\,dt\,.
\]

The irreducible components of $Sing_2(\fa_\om)$ are $V_1=(x=y=0)$, $V_2=(y=t=0)$, $V_3=(x^2+xz=t=0)$.
In the affine coordinate system $(z=1)$ we have 
\[
\om|_{z=1}=2xyt\,dx+t(x^2+2y)dy-3y(x^2+y)\,dt\,\,\implies
\]
\[
\mu(\fa_\om,V_1)=3\,\,,\,\,\mu(\fa_\om,V_2)=\mu(\fa_\om,V_3)=1\implies
\]
\[
d^2+d+1-\sum_{j=1}^3\mu(\fa_\om,V_j).deg(V_j)=7-6=1\,,
\]
so that $\G_{\fa_\om}$ is birational.

Another example is a foliation in $\L(1,3;3)$ whose Gauss map is birational, is defined in homogeneous coordinates by the form $\om=3\,f\,dz-z\,df$, where
$f(x,y,z,t)=xz^2+yzt+y^3$.

In this case, we have $Sing(\fa_\om)=(y=z=0):=V$ and $\mu(\fa_\om,V)=6$, so that its Gauss map, $\G_\om$, is birational according to Theorem \ref{p4}.
We have:
\[
\G_\om(x,y,z,t)=(-z^3,-z(3y^2+zt),xz^2+2yzt+3y^3,-yz^2)
\]
The inverse can be calculated as:
\[
\G_\om^{-1}(a,b,c,e)=(a^2c+2abe-3e^3,-a^2e,-a^3,a(3e^2-ab)).
\]

\noindent The foliation associated with this map is $(\G_\om)_*(\fa_\om)$ and it has the rational first integral:

$$\frac{a^2c+abe-e^3}{a^3},$$

\noindent then, this is also a self-dual foliation.
\end{ex}

\vskip.1in
Next, we will consider foliations defined by closed forms.
Among them, we have logarithmic foliations on $\p^n$, say $\fa\in\L(p_1,...,p_k;n)$. They are defined in homogeneous coordinates $z\in\C^{n+1}$ by a logarithmic 1-form of the type
\[
\om_\fa=f_1...f_k\sum_{j=1}^k\la_j\frac{df_j}{f_j}\,,
\]
where $f_j\in\C[z_0,...,z_n]$ is homogeneous of degree $p_j$ and $\sum_{j=1}^kp_j\la_j=0$ (see \cite{ca} and \cite{cl}).  If $\om_\fa$ has no zero divisor, that is $cod(Sing(\om))\ge2$, then we have $dg(\fa)=p_1+...+p_k-2$.

In this case, we have the following.

\begin{prop}\label{p5}
Let $\fa$ be as above, given by $\om_\fa=f_1...f_k\sum_{j=1}^k\la_j\frac{df_j}{f_j}$, $\sum_j\la_jp_j=0$.
Assume that:\\
1. $\om_\fa$ has no zero divisor.\\
2. The hypersurfaces $V_i$ and $V_j$ of $\p^n$, given by $V_i=(f_i=0)$ and $V_j=(f_j=0)$, are generically transverse, for all $i<j$.

If $n=3$ and $\G_\fa$ is dominant then generically $\G_\fa$ is $\ell$ to one, where
\begin{equation}\label{eq4}
\ell=tdg(\G_\fa)=\sum_{j=1}^kp_j^2+\sum_{i<j}p_ip_j-3\sum_{j=1}^kp_j+3\,.
\end{equation}
\end{prop} 

In the statement, by generically transverse we mean that the set of tangencies between $V_i$ and $V_j$ has codimension $>2$.

As a consequence, we have:

\begin{cor}\label{c15}
Suppose that the Gauss map of the generic member of $\L(p_1,...,p_k;3)$ is birational. Then either $k=4$ and $p_1=...=p_4=1$, or $k=2$ and $(p_1,p_2)=(1,2)$.
\end{cor}

\begin{rem}
\rm We would like to observe that a foliation $\fa\in\L(1,1,1,1;3)$ with Gauss map birational can be defined in some homogeneous coordinate system $(x,y,z,t)\in\C^4$ by the logarithmic form
\[
\te=\a\frac{dx}{x}+\be\frac{dy}{y}+\g\frac{dz}{z}+\d\frac{dt}{t}\,,
\]
whereas if $(p_1,p_2)=(1,2)$ then the foliation has a rational first integral of the form $Q/L^2$ where $deg(Q)=2$ and $deg(L)=1$.

Let us prove the first assertion. Since $p_j=1$, $1\le j\le4$, the foliation can be defined in homogeneous coordinates by
\[
\te=\a\frac{d\ell_1}{\ell_1}+\be\frac{d\ell_2}{\ell_2}+\g\frac{d\ell_3}{\ell_3}+\d\frac{d\ell_4}{\ell_4}
\]
where $\ell_j\in\C[z_1,z_2,z_3,z_4]$ has degree one, $1\le j\le4$.
Since the Gauss map is dominant we must have
\[
d\ell_1\wedge d\ell_2\wedge d\ell_3\wedge d\ell_4=c.dz_1\wedge dz_2\wedge dz_3\wedge dz_4\,,
\]
where $c\in\C^*$. This implies that $\ell:=(\ell_1,\ell_2,\ell_3,\ell_4)\colon\C^4\to\C^4$ is an automorphism, so in some coordinate system $(x,y,z,t)\in\C^4$, we have $(\ell_1,\ell_2,\ell_3,\ell_4)=(x,y,z,t)$.
\end{rem}

\begin{ex}\label{exdg3}
\rm As an application of Theorem \ref{p4} and of proposition \ref{p5}, we show below that 
$\n(3,3)=\{0,1,...,8,9\}.$ 
\\
1. In the case $tdg(\G_\fa)=1$ we take the foliation with first integral $F_2(x,y,z,t)=\frac{(xz+yt)y^2+x^4}{y^4}$, like in remark \ref{ex5}. As we have seen, this foliation has degree three and its Gauss map is birational.\\
2. The reader can verify directly using (\ref{eq4}) of proposition \ref{p5} that a generic foliation in $\L(1,1,1,1,1;3)$ has topological degree three.\\
3. A generic foliation in $\L(2,1,1,1;3)$ has topological degree four.\\
4. A generic foliation in $\L(2,2,1;3)$ has topological degree five.\\
5. A generic foliation in $\L(3,1,1;3)$ has topological degree six.\\
6. A generic foliation in $\L(3,2;3)$ has topological degree seven.\\
7. A generic foliation in $\L(4,1;3)$ has topological degree nine.\\
8. The case of $tdg(\G_\fa)=2$ can be realized by a foliation with first integral $f(x,y,z,t)=[z^3t+(yt-x^2)^2]/t^4$.
This foliation is represented by the form
\[
\om=4xt(ty+x^2)dx+2t^2(ty+x^2)dy+3z^2t^2\,dz-[3z^3t+2(ty+x^2)(ty+2x^2)]dt.
\]
Here, we have $Sing_2(\fa)=V_1\cup V_2$, $V_1=(x=t=0)$ and $V_2=(z=yt+x^2=0)$, $\mu(\fa,V_1)=7$ and $\mu(\fa,V_2)=2$, so that $\sum_j\mu(\fa,V_j)deg(V_j)=11$ and $tdg(\G_\fa)=13-11=2$.\\
9. The case $tdg(\G_\fa)=8$ is more involved. Let
\[
Q(x,y,z,t)=t^2(x^2+y^2+z^2)+t(x^3+y^3+z^3)+x^2(x^2+yz)\,.
\]

The foliation $\fa$ with first integral $F=Q/t^4$ has $tdg(\G_\fa)=8$. Let us prove this fact.

First of all, the foliation is represented in homogeneous coordinates by $\om=tdQ-4Qdt$ and in the affine system $(t=1)$ by $dQ|_{(t=1)}:=Adx+Bdy+Cdz$, where
\[
A=x(2+3x+4x^2+2yz)\,\,,\,\,B=2y+3y^2+x^2z\,\,\text{and}\,\,C=2z+3z^2+x^2y.
\]
This implies that $Sing_2(\fa)=(t=x=0)\cup(t=x^2+xy=0)$, because $dim(A=B=C=0)=0$.
We have done this computation and we have found 14 solutions for the system.

Let $V_1=(t=x=0)$ and $V_2=(t=x^2+xy=0)$. We assert that $\mu(\fa,V_1)=3$ and $\mu(\fa,V_2)=1$.
Let $\Si=(y=z=1)$. Note that $\Si$ is transverse to both varieties $V_1$ and $V_2$. 
Moreover,
\[
F|_\Si=\frac{t^2(x^2+2)+t(x^3+2)+x^2(x^2+1)}{t^4}\,\implies
\]
\[
\om|_\Si=tx(2+4x^2+3xt+2t^2)dx-(2t^2(x^2+2)+3t(x^3+2)+4x^2(x^2+1))dt
\]
\[
:=tx\,M(x,t)dx+N(x,t)dt.
\]
We have $\Si\cap V_1=(0,0)=0$ and $\Si\cap V_2=(t=0,x=i)\cup(t=0,x=-i)=\{p,q\}$. Since $M(0,0)\ne0$, we have
\[
\mu(\fa,V_1)=\mu_0(tx\,M,N)=\mu_0(t,x^2)+\mu_0(x,t)=2+1=3\,.
\]
Similarly, it can be checked that $\mu(\fa,V_2)=\mu_p(tx\,M,N)=\mu_q(tx\,M,N)=1$, and so $\sum_j\mu(\fa,V_j)deg(V_j)=5$.
This implies that $tdg(\G_\fa)=13-5=8$.
\end{ex}

\begin{rem}
\rm From a foliation $\fa$ on $\p^3$, of degree $d$, whose Gauss map $\G_\fa$ is birational, we can produce a family of rational curves on $\p^3$ of degree $d+1$ as we explain next.
Let
\[
\L=\{\ell\in G_1(\p^3)\,|\,\ell\cap Sing(\fa)=\emp\,\,\text{and $\ell$ has $d$ tangencies with $\fa\}$}\,,
\]
where $G_1(\p^3)$ denotes the Grassmannian of lines of $\p^3$.
Note that $\L$ is a Zariski open and dense subset of $G_1(\p^3)$.

Fix $\ell\in\L$ and denote as $\P_\ell$ the pencil of planes of $\p^3$ that contain the line $\ell$. 
Now, we observe the following facts:\\
1. A plane $P\in\P_\ell$, transverse to $Sing(\fa)$, has exactly one Morse tangency with $\fa$, because the Gauss map is birational.
This Morse tangency is in fact $\G_\fa^{-1}(P)$, which we will denote as $\M_\ell(P)$.\\
2. Let $T=\{P\in\P_\ell\,|\,\,P$ is not transverse to $Sing(\fa)\}$. Then, $T$ is finite.\\
3. The set $\P_\ell\setminus T$ is isomorphic to the complement of a finite subset of $\p^1$.\\
4. The map $\M_\ell\colon \P_\ell\setminus T\to \p^3$ is algebraic.

In particular, if we fix $\ell\in\L$ then the map $\M_\ell$ extends analytically to $\P_\ell$, so that the image $C_\ell:=\M_\ell(\P_\ell)$ is a rational curve of $\breve\p^3$.

It remains to prove that $C_\ell$ has degree $d+1$. Let us fix a point of tangency between $\ell$ and $\fa$, say $p\in\ell$.
Let $P=\G_\fa(p)\in\breve\p^3$.

\begin{claim}
The plane $P=\G_\fa(p)$ is in the pencil $\P_\ell$. In particular, $\ell\sub P$.
Moreover, $\M_\ell(P)=p$ and the curve $C_\ell$ is tangent to $P$ with tangent space $T_pC_\ell$ transverse to $T_p\ell$ (in the plane $P$).
\end{claim}

{\it Proof.}
Since $\ell$ and $P$ are tangent to $\fa$ at $p$, and $p$ is a Morse tangency, it is clear that
$\ell\sub P$. Let $f\in\O_p$ be a local first integral of $\fa$ near $p$, with $f(p)=0$. Choose an affine coordinate system $(x,y,z)\in\C^3$ such that $P=(z=0)$, $\ell=(y=z=0)$, and $p=(0,0,0)$. Since $f$ is a first integral of $\fa$, and $p$ is the Morse tangency of $P$ with $\fa$, we can assume that
$f(x,y,z)=z+q(x,y,z)+h.o.t$, where $q(x,y,0)$ is a non degenerate quadratic form and $h.o.t$ are terms of order $\ge3$.

After a linear change of variables, we can assume that $q(x,y,0)=x^2+y^2$, so that $f(x,y,z)=z+x^2+y^2+axz+byz+cz^2+h.o.t$.
Since $\ell=(y=z=0)$, the pencil $\P_\ell$ can be parametrized near $z=0$ as $\a\in\C\mapsto z=\a\,y$, so that
\[
f|_{z=\a y}(x,y,\a y)=\a y+x^2+y^2+a\a xy+b\a y^2+c\a^2 y^2+h.o.t
\]
This implies that $\M_\ell(z=\a y)$ can be calculated as the zeroes of 
\[
F(x,y,\a):=\left(\frac{\pa}{\pa x}f(x,y,\a y),\frac{\pa}{\pa y}f(x,y,\a y)\right)=
\]
\[
=\left(2x+a\a y,\a+a\a x+2(1+b\a+c\a^2)y\right)+h.o.t
\]

Using the implicit function theorem, we can solve $F(x,y,\a)=(0,0)$ as 
\[
\a\mapsto (X(\a),Y(\a))=\big(o(\a),-\frac{\a}{2}+o(\a)\big)\,.
\]
Since we are in the pencil $z=\a y$, we get
\[
\M_\ell(z=\a y)=(o(\a),-\frac{\a}{2}+o(\a),-\frac{\a^2}{2}+o(\a^2))\,\implies
\]
$T_pC_\ell=(x=z=0)$ which is transverse to $T_p\ell$, in the plane $P$.\qed
\vskip.1in
Now, fix another point $q$ of tangency between $\ell$ and $\fa$. Note that $\ell\sub\G_\fa(q)\ne\G_\fa(p)$.
Since $T_qC_\ell$ is contained in $\G_\fa(q)$ and transverse to $\ell$, we conclude that $T_qC_\ell$ is transverse to $P=\G_\fa(p)$.
This implies that $C_\ell$ cuts $P$ exactly in the $d$ points of tangency of $\ell$ with $\fa$, but in just one of these points it is not transverse to $P$. Therefore, the degree of $C_\ell$ is $d+1$.
\end{rem}

Finally, we will comment something about the degree of the Gauss map. Let $G\colon\p^n\dasharrow\p^n$, $G[z]=[G_0[z],...,G_n[z]]$ be a birational map on $\p^n$. The degree of $G$, $deg(G)$, is by definition the degree of the polynomials $G_j$. The bidegree of $G$ is the pair $(deg(G),deg(G^{-1}))$. If $n=2$ it is known that $deg(G)=deg(G^{-1})$ (see \cite{sc}), but in dimension $n\ge3$ this is not true in general (see \cite{pan}).

For the case of foliations $\fa \in \fol(d;n)$, a corollary of Theorem 1 of \cite{tj} says that if $n=3$ and the Gauss map $\G_\fa$ is birational then: $deg(\G_\fa)=deg(\G_\fa^{-1})$.
 In a personal communication, Fassarella has shown us the following example of a foliation in $\fol(2;4)$ with birational Gauss map and such that the degree of the inverse is not the same.
 
 Let $f(x,y,z,y,w)=z^3+yzt+xt^2+t^3$. Consider the foliation $\mathcal{F}$ of degree $2$ on  $\mathbb{P}^4$ given by the 1-form:

\begin{align*}
\omega=wdf-3fdw&=t^2wdx+ztwdy+w(yt+3z^2)dz+w(2xt+yz+3t^2)dt\\
&-3(z^3+yzt+xt^2+t^3)dw\\
&=Adx+Bdy+Cdz+Ddt+Edw.
\end{align*}

Its Gauss map is:

\begin{align*}
\mathcal{G}_{\mathcal{F}}: \mathbb{P}^4 &\dashrightarrow \mathbb{P}^4\\
(x:y:z:t:w) &\mapsto (A:B:C:D:E),
\end{align*}

\noindent which has degree $3$. We can see that $\mathcal{G}_{\mathcal{F}}$ is birational with inverse:

\begin{align*}
\mathcal{G}_{\mathcal{F}}^{-1}: \mathbb{P}^4 &\dashrightarrow \mathbb{P}^4\\
(A:B:C:D:E) &\mapsto \Big(E\big(DA^2-B(CA-3B^2)-3A^3\big): 2EA(CA-3B^2):\\
&\hspace{2cm} 2A^2BE:2A^3E:3A(B^3-ABC+A^3-A^2D)\Big),
\end{align*}

\noindent which has degree $4$. 

\subsection{Problems}

Here we state some problems that arise naturally from the results and examples discussed above.

\begin{prob}
\rm   Among the foliations $\fa \in \fol(d;n)$, for $n \geq 4$, whose Gauss map  $\G_\fa$ is birational, find properties of those for which 
$deg(\G_\fa) \neq deg(\G_\fa^{-1})$.
\end{prob}

\begin{prob}
    \rm When the Gauss map of a foliation is not birational, then the image of the foliation by its Gauss map defines a $k$-web, where $k$ is the topological degree of the foliation. Then, the problem would be to find a relationship between the degree of the foliation and the degree of the $k$-web that its Gauss map defines. 
\end{prob}

A good reference for the theory of k-webs is the book \cite{pereira-pirio}. Here it is appropriate to include the following comment, which has been suggested by the referee:
If $\fa$ is an arbitrary codimension one foliation on $\p^n$ then $(\G_{\fa})_* \fa$ is the dual web or Legendre transform of $\fa$ on $\breve\p^n$.
\\

We have seen in remark \ref{rem15} that the hypothesis of corollary \ref{c12} on foliations of $\p^n$ is not necessary for a Gauss map to be birational, although it is sufficient. 

\begin{prob}
\rm Are there conditions on a foliation of $\p^n$, $n\ge4$, that are necessary and sufficient for the Gauss map to be birational?
\end{prob}

A related question is:

\begin{prob}
\rm Classify all the foliations of degree two on $\p^n$, $n\ge3$, for which the Gauss map is birational. It is reasonable to start with this classification for $n=3$ (see section 4.2, in particular).
\end{prob}

\begin{prob}
\rm Are there other magic examples (as in example \ref{exAB}), that is, are there other rational maps $\Phi: \mathbb{P}^3 \dasharrow \mathbb{P}^2$, such that the pull-back by $\Phi$ of every foliation on $\mathbb{P}^2$ has birational Gauss map? 
\end{prob}

We have seen in Theorem \ref{pM} that a foliation $\fa$ of degree $d$ of $\p^3$ such that $tdg(\G_\fa)=d^2$ has a first integral. Then we can ask:

\begin{prob}
\rm Let $\fa\in\fol(d;3)$ such that $tdg(\G_\fa)=d^2-1$. Is it true that $\fa$ has a rational first integral?

\end{prob}

\begin{prob} A foliation $\fa \in\fol(d;n)$ where $n\geq 4$, and such that $tdg(\G_\fa)=d^{n-1}$ has a rational first integral?
\end{prob}

In corollary \ref{c15} we have classified the generic logarithmic foliations on $\p^3$ whose Gauss map is birational.

\begin{prob}
\rm Is there such a classification in the case of $\p^n$, $n\ge4$?
\end{prob}
 
We have seen in examples \ref{dg2} and \ref{exdg3} that if $d=2$ or $d=3$ then $\n(d,3)=\{0,1,...,d^2\}$.
A natural problem is the following.

\begin{prob}
\rm Is $\n(d,3)=\{0,1,2,...,d^2\}$ for all $d\ge4$?
\end{prob}

\section{Proofs}\label{ss2}

\subsection{Proof of Theorem \ref{pM}}\label{sspM}
Let $\fa$ be a foliation on $\p^2$, defined in some affine chart $\C^2\sub\p^2$ by a polynomial vector field $X$.
Let $p\in Sing(\fa)\cap\C^2$, so that $X(p)=0$. We say that $p$ is a non-degenerate singularity if $det(DX(p))\ne0$.
We say that $p$ is a singularity with trace zero if $tr(DX(p))=0$.

The proof of Theorem \ref{pM} is based on the following:
\begin{prop}\label{pr2}
\rm A foliation on $\p^2$ of degree $d$ has at most $d^2$ non-degenerate singularities with trace zero.
Moreover, if $\fa\in \mathbb{F}ol(d;2)$ has $d^2$ non-degenerate singularities with trace zero then it has a first integral of the form $F/L^{d+1}$ where $deg(F)=d+1$ and $deg(L)=1$.
\end{prop}

{\it Proof.}
Let $\fa$ be a foliation of degree $d$ on $\p^2$, having $d^2$ non-degenerate singularities with trace zero. Let $\C^2\sub\p^2$ be an affine coordinate system such that $Sing(\fa)\sub\C^2$. In this case, the line at infinity $\ell_\infty=\p^2\setminus\C^2$ is not $\fa$-invariant.

Note that $\fa$ can be defined in these coordinates by a polynomial vector field $X=P(x,y)\frac{\pa}{\pa x}+Q(x,y)\frac{\pa}{\pa y}$, or equivalently by the form $\om=P(x,y)\,dy-Q(x,y)\,dx$, where:\\
1. $P(x,y)=p(x,y)+x\,g(x,y)$ and $Q(x,y)=q(x,y)+y\,g(x,y)$.\\
2. $max(deg(p),deg(q))=d$ and $g$ is homogeneous of degre $d$.  

We can take the polynomials like this because the line $\ell_\infty=(z=0)$ is not $\fa$-invariant, then $g\ne0$ and $z \nmid g(x,y)$; therefore we can consider the homogeneous part of $g$ of degree d.
The condition $max(deg(p),deg(q))=d$,  comes from the fact that $\ell_\infty$ does not contain singularities of $\fa$, then $z \nmid p(x,y)$ or $z \nmid q(x,y)$.
\\

Let $p_1,...,p_{d^2}$ be the singularities, with linear part $L_1,...,L_{d^2}$ respectively, where $det(L_i)\ne0$ and $tr(L_i)=0$, $1\le i\le d^2$.

Note that
\[
tr(L_i)=\frac{\pa P}{\pa x}(p_i)+\frac{\pa Q}{\pa y}(p_i)=0
\]
and $d\om=\De\,dx\wedge dy$, where
\[
\De=\frac{\pa P}{\pa x}+\frac{\pa Q}{\pa y}
\]
so that $d\om(p_i)=0$, $1\le i\le d^2$.
A straightforward computation gives
\[
\De=\frac{\pa p}{\pa x}+\frac{\pa q}{\pa y}+2g+x\frac{\pa g}{\pa x}+y\frac{\pa g}{\pa y}=\frac{\pa p}{\pa x}+\frac{\pa q}{\pa y}+(d+2)g\,,
\]
because $g$ is homogeneous of degree $d$.

\begin{claim}
We assert that $(P=0)$ and $(\De=0)$ (resp. $(Q=0)$ and $(\De=0)$) meet transversely in $\C^2$ exactly at the points $p_1,...,p_{d^2}$.
\end{claim}

{\it Proof.}
Let $V:=\ov{(P=0)}$ and $W:=\ov{(\De=0)}$ be the closures in $\p^2$ of $(P=0)$ and $(\De=0)$, respectively.
Since $deg(P)=d+1$ and $deg(\De)=d$ the intersection $V\cap W$ contains $d^2+d$ points counted with multiplicity.
Among these points we have the $d^2$ points $p_1,...,p_{d^2}$. The other points of intersection are at the line of infinity $\ell_\infty$ and correspond to the $d$ linear factors of $g$ (not necessarily different).
This implies the claim.\qed 

At this point, we can assert that $p_1,...,p_{d^2}$ are all singularities of $\fa$ with trace zero because $\De=0$ contains all such singularities. Therefore, the maximum number of non-degenerate singularities with zero trace is $d^2$: there is no other such singularity.
\vskip.1in
We have proved above that $V$ and $W$ meet transversely in $\C^2$ at the points $p_i$, $1\le i\le d^2$, and at $\ell_\infty$ at the points corresponding to $g=0$. Since $Q(p_i)=0$, $1\le i\le d^2$, and it meets $\ell_\infty$ and $W$ at the points corresponding to the $d$ linear factors of $g$, we obtain that $Q\in\left<P,\De\right>$, the ideal generated by $P$ and $\De$.

Therefore, we can write $Q=\a\,P+f\,\De$, where $\a\in\C^*$ and, since $deg(f)\leq 1$ in addition to the fact that $gcd(P,Q)=1$, we have that $deg(f)= 1$. In particular, if $p\in Sing(\fa)\setminus\{p_1,...,p_{d^2}\}$ then its trace is non zero and $f(p)=0$, because $P(p)=Q(p)=0$, but $\De(p)\ne0$.

\vskip.1in
We have concluded that the $d+1$ singularities (counted with multiplicities) of $\fa$, which are not in the set $\{p_1,...,p_{d^2}\}$, must be in the line $(f=0)$.
Since $\fa$ has degree $d$ and $\ov{(f=0)}$ contains $d+1$ singularities of $\fa$ the line $L:=\ov{(f=0)}$ must be $\fa$-invariant (see \cite{br1}).
\vskip.1in
Now, we change the affine coordinate system.
Let $\C^2=\p^2\setminus L$. In this new affine coordinate system, the foliation can be represented by a form $\eta=A(x,y)dx+B(x,y)dy$, where $max(deg(A),deg(B))=d$, because the line at infinity $\ell_\infty=L$ is $\fa$-invariant.
In fact, since $\{p_1,...,p_{d^2}\}\sub\C^2$ and $A(p_i)=B(p_i)=0$ for all $i=1,...,d^2$, we must have $deg(A)=deg(B)=d$.

\begin{claim}
$d\eta=0$. In particular, we have $\eta=dF$, where $deg(F)=d+1$.
\end{claim}

{\it Proof.}
Since $A(p_i)=B(p_i)=0$, $1\le i\le d^2$, and $deg(A)=deg(B)=d$, the curves $(A=0)$ and $(B=0)$ meet transversely at exactly the $d^2$ points $p_i$. In particular the ideal $\I=\left<A,B\right>$, generated by $A$ and $B$, has the following property: if $H\in\C[x,y,z]$ is such that $H(p_i)=0$, $1\le i\le d^2$ then $H\in\I$. In particular, if $H\in\I$ and $H\ne0$ then $deg(H)\ge d$.

Now, since $p_1,...,p_{d^2}$ are of trace zero for the dual vector field, we must have $d\eta(p_i)=0$, $1\le i\le d^2$.
Note that $d\eta=G\,dx\wedge dy$, where $G=\frac{\pa B}{\pa x}-\frac{\pa A}{\pa y}$, so that $G(p_i)=0$, $1\le i\le d^2$.
Therefore $G\in\I$, but this implies $G=0$ because $deg(G)<d$.

Hence $\eta=dF$ where $deg(F)=d+1$. This implies that $\fa$ has the first integral $F/L^{d+1}$.
This finishes the proof of the claim and of proposition \ref{pr2}. \qed
\vskip.1in
Let us finish the proof of Theorem \ref{pM}. Suppose that a foliation $\fa$ of degree $d$ on $\p^3$ has a first integral of the form $F/L^{d+1}$, where $F\in\C[x,y,z,t]$ is homogeneous of degree $d+1$.
If $F$ is non singular then for a generic hyperplane $\p^2\simeq\Si\sub\p^3$ the restriction $F|_\Si$ has $d^2$ Morse points, so that the fiber of the Gauss map associated has $d^2$ points.
\vskip.1in
Conversely, let $\fa$ be a foliation of $\p^3$ such that the generic fiber of the Gauss map has $d^2$ points.
In particular, if $\Si\in\breve\p^3$ is a generic hyperplane then $\fa|_\Si$ has $d^2$ singularties with trace zero, so that by  proposition \ref{pr2}, $\fa|_\Si$ has a first integral of the form $f/\ell^{d+1}$. This implies that $\fa$ has a first integral of the form $F/L^{d+1}$ (see \cite{cerveau-moussu}).
This finishes the proof of theorem \ref{pM}.
\qed

\subsection{Proof of Proposition \ref{p1}}\label{ss21}
Let $\Phi=\G^{-1}\colon \p^n\dasharrow\p^n$ be the inverse of $\G$.
Then there are polynomial maps $\wh\G\colon \C^{n+1}\to\C^{n+1}$ and $\wh\Phi\colon\C^{n+1}\to \C^{n+1}$ such that
$\Pi\circ\wh \G=\G\circ\Pi$ and $\Pi\circ\wh\Phi=\Phi\circ\Pi$, so that  $\G[z_0:...:z_n]=[\wh\G_0(z):...:\wh\G_n(z)]$
and $\Phi[w_0:...:w_n]=[\wh\Phi_0(w):...:\wh\Phi_n(w)]$.

Note that the components of $\wh\G$ are homogeneous polynomials of the same degree $deg(\wh\G)$.
Similarly, the components of $\wh \Phi$ have the same degree $deg(\wh \Phi)$.

The indetermination set of $\wh\Phi$ has codimension $>1$, and this implies that the components of $\wh\Phi$ have no common factor.

Now, from $\Phi\circ \G=I$, we must have $\wh\Phi\circ\wh \G=f.I$, where $f$ is a homogeneous polynomial of degree $deg(\wh\Phi).deg(\wh \G)-1$. Note that $deg(\wh \G)>1$ and $deg(\wh \Phi)>1$, because otherwise $\G$ and $\Phi$ would be automorphisms of $\p^n$. In particular, $f$ is non constant, and $(f=0)$ defines in homogeneous coordinates a hypersurface, say $V$.
We assert that $\G(V)$ has codimension $\ge2$. 

In fact, from $\wh\Phi\circ\wh \G=f.I$ we get $\wh\Phi\circ\wh \G(f=0)=0$ and this implies that $\wh\Phi_j(\wh \G(f=0))=0$, $1\le j\le n$. Since the components of $\wh\Phi$ have no common components, we must have $cod(\wh \G(f=0))\ge2$, which implies that $cod(\G(V))\ge2$.

It remains to prove that if $h$ is not a factor of $f$ then it is not contracted by $\G$. 
In this case, we have that $(f=0)\cap(h=0)$ has codimension two and the set
\[
U:=\{[z]\in \p^n\,|\,h(z)=0\,\,\text{and}\,\,f(z)\ne0\}
\]
has dimension $n-1$ and is Zariski dense in $(h=0)$. 
Let $[z]\in U$. In homogeneous coordinates we have
\[
[z]\in U\,\implies\,\wh\Phi\circ\wh \G(z)=f(z).\,z\ne0\,\implies
\]
\[
dim(\Phi\circ \G(U))=n-1\,\implies\,dim(\G(U))=n-1\,,
\]
so that, $dim(\G(h=0))=n-1$ and $(h=0)$ is not contracted, as we wished.
\qed

\subsection{Proof of Proposition \ref{p2}}\label{ss22}
Let $\eta=\sum_{j=0}^nA_j(z)dz_j$ be an integrable form representing $\fa$ in homogeneous coordinates, where $A_0$, $A_1$,...,$A_n$ are homogeneous polynomials of the same degree satisfying the relation $i_R\eta=\sum_{j=0}^nz_jA_j(z)=0$.
In this case, we have
\[
\G_\fa[z]=[\wh\G(z)]=[A_0(z):...:A_n(z)]=[w_0:...:w_n]\,,
\]
which by hypothesis is birational.
Let $\Phi[z]=[\wh\Phi(z)]=[\wh\Phi_0(z):\wh\Phi_1(z):...:\wh\Phi_n(z)]$ be the inverse of $\G_\fa$, so that $\wh\G_\fa\circ\wh\Phi(w)=g(w).w$, where $g\in\C[z_0,...,z_n]$ is non constant.
Then
\[
\wh\Phi^*(\eta)=\sum_{j=0}^nA_j(\wh\Phi(w))\,d\wh\Phi_j=g(w)\sum_{j=0}^nw_j\,d\wh\Phi_j\,,
\]
so that $\G_{\fa*}(\fa)$ is represented by the form $\te:=\sum_{j=0}^nw_j\,d\wh\Phi_j$.
We assert that $\te=-\sum_{j=0}^n\wh\Phi_j(w)\,dw_j$.

In fact, composing the relation $\sum_{j=0}^nz_j\,A_j(z)=0$ with $\wh\Phi$ we get
\[
0=\left(\sum_{j=0}^nz_j\,A_j(z)\right)\circ\wh\Phi(w)=\sum_{j=0}^n\wh\Phi_j(w)\,A_j\circ\wh\Phi(w)=g(w)\sum_{j=0}^nw_j\wh\Phi_j(w)\implies
\]
\[
\sum_{j=0}^nw_j\wh\Phi_j(w)=0\,\implies\,\te=\sum_{j=0}^nw_j\,d\wh\Phi_j=-\sum_{j=0}^n\wh\Phi_j(w)\,dw_j
\]
Since $\wh\Phi$ is birational, we obtain $\G_{(\G)_*(\fa)}=\G_{\fa}^{-1}$.
\qed

\subsection{Proof of Theorem \ref{tnp}}\label{ss?}
Let $\fa=F^*(\G)$, where $F\colon\p^3\dasharrow\p^2$ has non planar fibers, except for a finite number of them.
We assume that the Gauss map $G_\fa$ is birational and that $G_\fa$ contracts a hypersurface $S$, not $\fa$-invariant. Remark that $S$ is not a plane because, using the condition that $G_{\fa}(S)$ is a point, we would have that $S$ is invariant. Then, the idea is to prove that $\fa$ has infinitely many algebraic leaves. By a theorem of Darboux, this implies that $\fa$ has a rational first integral (see \cite{j}). Let us prove this fact.

Since $S$ is not contracted to a point, $C:=G_\fa(S)$ is a curve of $\breve\p^3$.
In particular, there is a finite subset $A\sub C$ such that for any $m\in C':=C\setminus A$, $G_\fa^{-1}(m)$ is a curve of $S$.
Note that $G_\fa^{-1}(m)$ is a plane curve because it is contracted to a point by $G_\fa$. Moreover, $G_\fa^{-1}(m)$ is contained in a leaf of $\fa$, by the definition of $G_\fa$. 
 
Since the generic fibers of $F$ are not plane curves, taking a generic $m\in C'$, we can suppose also that $G_\fa^{-1}(m)$ is not a fiber of $F$, so that $\wh{C}$, the closure of $F(G_\fa^{-1}(m))$, is a curve of $\p^2$, contained in a leaf of $\G$.

\[
(\p^2,\wh{C})\underset{F}\longleftarrow(\p^3,S\sup G_\fa^{-1}(m))\underset{G_\fa}\longrightarrow(\breve\p^3,C\ni m)
\]

Since the fibers of $F$ are not plane curves $G_\fa^{-1}(m)\sub F^{-1}(\hat{C})$, but $G_\fa^{-1}(m)\ne F^{-1}(\hat{C}) $, so that the closure of $F^{-1}(\hat{C})$ is an algebraic surface of $\p^3$. Since $\fa=F^*(\G)$ we get the leaf of $\fa$ containing $G_\fa^{-1}(m)$ is algebraic. 

Now, if $m'\in C'$ and $m'\ne m$ then $G_\fa^{-1}(m')\ne G_\fa^{-1}(m)$, because these curves are contracted to a point by $G_{\fa}$. The intersections $G_\fa^{-1}(m) \cap S$ define a foliation of the surface $S$ by algebraic curves. If $\mathcal{L}$ is a leaf of $\fa$ containing
$G_\fa^{-1}(m)$, $\mathcal{L}$ intersects $S$ in at most a countable number (in fact finite, since it is algebraic) of connected components. Therefore, for generic $m$ and $m'$, the corresponding leaves of $\fa$ are distinct. Consequently, $\fa$ has infinitely many algebraic leaves.
This proves the result.\qed

\vskip.1in
\subsection{Proofs of Theorem \ref{p4} and of Corollaries \ref{c}, \ref{c12} and \ref{c14}.}\label{ss24}

\subsubsection{Proof of Theorem \ref{p4}}
Let $\fa$ be a foliation on $\p^3$ with $Sing_2(\fa)=V_1\cup...\cup V_k$ and $deg(\fa)=d$. Let $\p^2\simeq\Si\sub\p^3$ be a 2-plane in general position with $\fa$, in the sense that $\Si$ intersects transversely $Sing(\fa)$. Since $dim(\Si)=2$, this implies that:\\
1. $\Si\cap Sing(\fa)\sub Sing_2(\fa)$.\\
2. For all $i=1,...,k$, $\Si\cap V_i$ contains exactly $deg(V_i)$ points, say $p_{ij}$, $1\le j\le deg(V_i)$.

Since $p\in V_i\mapsto\mu(\fa,V_i,p)$ is constant in a Zariski dense subset of $V_i$, we can assume also that:\\
3. $\mu(\fa,V_i,p_{ij})=\mu(\fa,V_i)$ for all $1\le j\le deg(V_i)$.\\
4. $cod(Sing(\fa|_\Si))\ge2$. In particular, all singularities of $\fa|_\Si$ are isolated.

In this case, we have (see \cite{ls}):
\[
\sum_{p\in Sing(\fa|_\Si)}\mu(\fa|_\Si,p)=d^2+d+1\,.
\]

Since $\Si\cap Sing(\fa)\sub Sing_2(\fa)$ we get
\[
\sum_{i=1}^k\mu(\fa,V_i).deg(V_i)=\sum_{i=1}^k\sum_{j=1}^{deg(V_i)}\mu(\fa|_\Si,p_{ij})\le d^2+d+1
\]

Let $m=(d^2+d+1)-\sum_{i=1}^k\mu(\fa,V_i).deg(V_i)$. In this case, if $\Si\simeq\p^2$ is a 2-plane of $\p^3$ transverse to $Sing_2(\fa)$ then the singularities of the restricted foliation $\fa|_\Si$ are of two kinds: the points of $\Si\cap Sing_2(\fa)$ and the tangencies between the leaves of $\fa$ and $\Si$. After a small perturbation of the plane $\Si$, we can assume that these tangencies are all of Morse type (see \cite{cl}). On the other hand, the number of singularities of $\fa|_\Si$ in $\Si\cap\bigcup_iV_i$, counted with multiplicities, is exactly 
\[
\sum_{p\in \Si\cap Sing_2(\fa)}\mu(\fa,V_i).deg(V_i)\,\implies\,
\]
the other singularities of $\fa|_\Si$ are the $m$ Morse tangencies (note that $m$ could be zero). Hence $\G_\fa^{-1}(\Si)$ contains $m$ points. \qed

\subsubsection{Proof of corollary \ref{c}}
First of all, let $p_1,...,p_N$, $N={d^2+d}$, be the singularities of $\fa_o|_\Si$ with generic linear part. Since $\fa_o|_\Si$ has degree $d$, it must have a total of $d^2+d+1=N+1$ singularities, counted with multiplicity. However, the hypothesis implies that $\mu(\fa_o|_\Si,p_j)=1$, $1\le j\le N$, so that $\fa_o$ must have another singularity, say $q$, with $\mu(\fa_o,q)=1$. Since $\G_{\fa_o}$ is dominant, the singularity $q$ is necessarily a Morse center and $q$ is a Morse tangency of $\fa_o$ with $\Si$. Now, we use the known fact that singularities with non degenerate linear part are stable by small deformations of the foliation on $\p^2$ (see section 2.4 and corollary 2.7 of \cite{ls}). This implies that:\\
1. The irreducible component $V_{oj}$ of $Sing_2(\fa_o)$ containing $p_j$ satisfies $\mu(\fa_o,V_{oj})=1$, so that
\[
\sum_{V\sub Sing_2(\fa_o)}\mu(\fa_o,V).deg(V)=d^2+d\,.
\]\\
2. When we deform $\fa_o$ inside the irreducible component $Z$ of $\fol(3,d)$, the deformed foliation, say $\fa\in Z$, must also satisfy (see \cite{ls}):
\[
\sum_{V\sub Sing_2(\fa)}\mu(\fa,V).deg(V)=d^2+d\,.
\]
This finishes the proof.\qed

\subsubsection{Proof of corollary \ref{c12}.}
Let $\fa$ and $V_1,...,V_k$ be as in the hypothesis of Corollary \ref{c12} and let $\G_\fa\colon \p^n\dasharrow\breve\p^n$ be its Gauss map. The idea is to prove that a generic hyperplane of $\p^n$ has just one tangency point with the foliation. This implies that $\G_\fa$ is birational.

First of all, if $\p^2\simeq\Si_2$ is a generic two-plane of $\p^n$ then it has exactly one Morse tangency point with $\fa$ at a point in $\Si_2\setminus Sing(\fa)$. 
This can be proved using the condition $\sum_{j=1}^k\mu(\fa,V_j).deg(V_j)=d^2+d$.
The proof is done in the same way as the proof of Theorem \ref{p4}.
We leave the details to the reader.

Now, since $\G_\fa$ is dominant, a generic $(n-1)$-plane of $\p^n$ has at least one tangency of Morse type with $\fa$ at a non singular point. Suppose by contradiction that there exists a hyperplane, say $\Si_{n-1}$ of $\p^n$ with at least two Morse tangencies with $\fa$, say $m_1$ and $m_2$.
In this case, let $\Si_2\sub\Si_{n-1}$ be a 2-plane containing the points $m_1$ and $m_2$.
Then it is clear that for a generic choice of such a 2-plane, $\Si_2$ would have tangencies of Morse type with $\fa$ at these points, a contradiction.
\qed

\subsubsection{Proof of corollary \ref{new}} Since $\p^m$ is linearly embedded in $\p^n$, the components of the singular set of the restriction $\fa_{|\p^m}$ satisfy the formula of the previous corollary, so we have the result.
\qed

\subsubsection{Proof of corollary \ref{c14}}
Let $\fa$ and $V_1$,...,$V_k$ be as before. Since $\mu(\fa,V_i)\ge1$ we have
\[
\sum_{j-1}^k\mu(\fa,V_j)\,deg(V_j)\ge \sum_{j-1}^kdeg(V_j)=d^2+d\qed
\]

\subsection{Proof of Theorem \ref{pmon}}\label{sspmon}
Let $\om$ be an integrable 1-form on $\C^4$ defining a monomial foliation on $\p^3$ of degree $d\geq 2$. We can write $\om=\sum_{i=1}^4A_i(z)dz_i$, where $A_{i's}$ are monomials of degree $\ge3$. Let
\[
A_i(z)=\la_iz_1^{a_{i1}}...z_4^{a_{i4}}\,,\,a_{ij}\in\Z_{\ge0}\,,\,i\in\{1,2,3,4\}\,.
\]
Setting $B_i=z_iA_i$, we can write
\[
\om=B_1\frac{dz_1}{z_1}+B_2\frac{dz_2}{z_2}+B_3\frac{dz_3}{z_3}+B_4\frac{dz_4}{z_4}
\]

The condition $i_R\om=0$ implies that $B_1+B_2+B_3+B_4=0$. 
Note that $B_i\ne0$, for all $i$. In fact, if $B_4=0$ for instance, then $\om=A_1dz_1+A_2dz_2+A_3dz_3$ and $\G_\om$ cannot be dominant. We also ask that the monomials $A_1$, $A_2$, $A_3$, $A_4$ have no variables in common to have codimension of the singular set $\geq 2$.

Observe that each $B_i$ is a monomial of degree $d+2$, with exponent vector $(a_{11}+1,a_{12},a_{13},a_{14})$, $(a_{11},a_{12}+1,a_{13},a_{14})$, $(a_{11},a_{12},a_{13}+1,a_{14})$, $(a_{11},a_{12},a_{13},a_{14}+1) \in \Z^4_{\geq 0}$, for $i=1,2,3,4$; respectively. Since the sum $B_1+B_2+B_3+B_4$ vanishes identically, the monomials must be grouped according to identical exponent vectors, with the coefficients in each group summing to zero. With this observation, only the following cases can occur.
\\

\noindent Case 1: All monomials in $\{B_1,B_2,B_3,B_4\}$
have exactly the same vector of exponents in $\Z^4_{\geq 0}$. In this case and since $A_i´s$ do not have variables in common, we have $B_i=\lambda_iz_1z_2z_3z_4$ for all $i=1,2,3,4$ and $\lambda_1+\lambda_2+\lambda_3+\lambda_4=0$. Therefore, the foliation $\fa_{\om}$ has degree $2$ and belongs to the logarithmic component $\L(1,1,1,1;3)$.
\\

\noindent Case 2: If the exponent vectors are not all equal, there cannot exist a single exponent vector differing from the others, because in this case the sum of the monomials is not zero. Therefore, with 4 exponent vectors, we only have the following possibility:
the set $\{B_1,B_2,B_3,B_4\}$ splits into exactly two disjoint subsets, each containing two monomials with identical exponents. And the sum of their coefficients is zero.
\\

Then, after a permutation of the indexes, we can assume that $B_1+B_2=0$ and $B_3+B_4=0$. Then we have $\om=B_1\Te_{12}+B_3\Te_{34}$, where $\Te_{ij}=\frac{dz_i}{z_i}-\frac{dz_j}{z_j}$.
Since $d\Te_{ij}=0$ and $\Te_{ij}=\frac{z_j}{z_i}d\left(\frac{z_i}{z_j}\right)$, the reader can check that integrability condition $\om\wedge d\om=0$ can be written as
\[
0=d\left(\frac{B_1}{B_3}\right)\wedge \Te_{12}\wedge\Te_{34}=0\,\implies\,d\left(\frac{B_1}{B_3}\right)\wedge d\left(\frac{z_1}{z_2}\right)\wedge d\left(\frac{z_3}{z_4}\right)=0.
\]
With this we conclude that equations $a_{11}+1-a_{31}=a_{32}-a_{12}=p$, $a_{13}-a_{33}-1=a_{34}-a_{14}=q$ are satisfied, to have
\[
\frac{B_1}{B_3}=\frac{\la_1}{\la_3}\left(\frac{z_1}{z_2}\right)^{p}\left(\frac{z_3}{z_4}\right)^{q}\,,
\]
where $\la \in \C^*$. Since $z_2 \vert B_1$ and $z_4 \vert B_3$, then $a_{12} >0$ and $a_{34} >0 $, this implies that $a_{32}=0$ and $a_{14}=0$; otherwise $z_2$ or $z_4$ divides $\omega$. Similarly, we can see that if $a_{11} >0$ or $a_{33}>0$ then $z_1 \vert A_1, A_2$ or $z_3 \vert A_3, A_4$, this implies $a_{31}=0$ or $a_{13}=0$, in the first case $p=-a_{12}=a_{11}+1 >0$ and, in the second case, $q=a_{34}=-a_{33}-1<0$. Both produce a contradiction; therefore, $a_{11}=a_{33}=0$. With this we obtain the following specific 1-form that defines the foliation:

$$\om=\lambda_1 z_1z_2^{a_{12}}z_3^{a_{13}}\Te_{12}+\lambda_3 z_1^{1+a_{12}}z_3z_4^{a_{13}-1}\Te_{34},$$

\noindent where $a_{12} \geq 1$ and $a_{13} \geq 2$. Moreover, we have
$p=-a_{12}=1-a_{31}<0$ and $q=a_{34}=a_{13}-1>0$. After a linear change of variables, we can assume $\frac{\la_1}{\la_3}=1$. With all this, we finally conclude that
\[
B_3^{-1}.\,\om=\frac{B_1}{B_3}\left(\frac{dz_1}{z_1}-\frac{dz_2}{z_2}\right)+\left(\frac{dz_3}{z_3}-\frac{dz_4}{z_4}\right)=
\]
\[
=\left(\frac{z_1}{z_2}\right)^{ p}\left(\frac{z_3}{z_4}\right)^{ q}\left[\frac{z_2}{z_1}d\left(\frac{z_1}{z_2}\right)\right]+\frac{z_4}{z_3}d\left(\frac{z_3}{z_4}\right)
\]
If we set $u=z_1/z_2$ and $v=z_3/z_4$ then we get
\[
vu^{1-p}B_3^{-1}\om= v^{q+1}du+u^{1-p}dv\,,
\]
as we wished.\qed

\subsection{Proofs of Proposition \ref{p5} and of Corollary \ref{c15}}\label{ss25}

\subsubsection{Proof of Proposition \ref{p5}}
We will apply Theorem \ref{p4}. The idea is to use that $Sing_2(\fa)=\bigcup_{1\le i<j\le k}V_{ij}$, where $V_{ij}=(f_i=f_j=0)$ (see \cite{ca}).
This is true by the assumption that $\om_\fa$ has no zero divisors.

Since the hypersurfaces $(f_i=0)$ and $(f_j=0)$ are generically transverse we have $\mu(\fa,V_{ij})=1$, and so
\[
\sum_{i<j}deg(V_{ij})\mu(\fa,V_{ij})=\sum_{i<j}\deg(V_{ij})=\sum_{i<j}p_i\,p_j\,. 
\]

Now, since $\om_\fa$ has no zero divisor, we have $deg(\fa)=\sum_{i=1}^kp_i-2$ (see example 2.5 of \cite{ls}) and
\[
deg(\fa)^2+deg(\fa)+1-\sum_{i<j}p_i\,p_j=\sum_{j=1}^kp_j^2+\sum_{i<j}p_ip_j-3\sum_{j=1}^kp_j+3:=\ell
\]

Note that necessarily $\ell>0$, for otherwise $\G_\fa$ is not dominant.
In particular, $\G_\fa$ is generically $\ell$ to one by Theorem \ref{p4}.
This proves Proposition \ref{p5}.\qed

\subsubsection{Proof of corollary \ref{c15}}
Let $\fa_\om\in\L(p_1,...,p_k;3)$:
\[
\om=\sum_{j=1}^k\la_j\frac{df_j}{f_j}\,,
\]
where $deg(f_j)=p_j$, $1\le i\le k$. If $\om$ is generic, then the hypersurfaces $(f_i=0)$ and $(f_j=0)$ are transverse, for all $i<j$, so that we can apply the proposition \ref{p5}.
In particular, $\G_\fa$ is birational if, and only if, 
\[
\sum_{j=1}^kp_j^2+\sum_{i<j}p_ip_j-3\sum_{j=1}^kp_j+3=1\,.
\]

Let 
\[
h_k(x)=f_k(x_1,...,x_k)=\sum_{j=1}^kx_j^2+\sum_{i<j}x_ix_j-3\sum_{j=1}^kx_j+3\,.
\]
We are interested in the values of $h_k$ in $\N^k\sub\R^k$.
Moreover, since $h_k$ is symmetric we will assume that $(x_1,...,x_k)\in\De$, where
\[
\De:=\{x\in\N^k\,|\,1\le x_1\le x_2\le...\le x_k\}\,.
\]

First of all, if $k=4$ then $h_4(1,1,1,1)=1$. Similarly, if $k=2$ then $h_2(1,2)=1$.
We will prove that these are the only possible solutions.

Note that if $x\in\De$ and $k\ge3$ then:
\[
\frac{\pa h_k}{\pa x_j}(x)=x_j+\sum_{i=1}^kx_i-3\ge k-2>0\,.
\]
Therefore, in $\De$, $h_k$ is strictly increasing with respect to $x_j$, for all $j$.
In particular, in $\De$ we have
\[
h_k(x)\ge h_k(1,...,1)=k+\frac{k(k-1)}{2}-3k+3=\frac{1}{2}(k-2)(k-3)):=g(k)\,.
\]
If $k\ge5$ and $x\in\De$, then $h_k(x)\ge g(5)=3$. Hence $k\le4$ and if $k=4$ we must have $(x_1,x_2,x_3,x_4)=(1,1,1,1)$.

It remains to prove that if $k\le 3$ then $k=2$ and $(x_1,x_2)=(1,2)$. 
In fact, we have $h_3(1,1,1)=g(3)=0$ and $h_3(1,1,2)=2$, so that
$h_3(x_1,x_2,x_3)\ne1$, for all $(x_1,x_2,x_3)\in\De$.
Therefore $k=2$ as we wished. Finally, if $x_1>1$ or $x_2>2$ then $h_2(x_1,x_2)> h_2(1,2)=1$, so that $(x_1,x_2)=(1,2)$ (see example \ref{ex2}).\qed
 
\section{Exceptional components}\label{ss3}

Before giving the results of this section we will describe the algorithm we have used to find the inverse of a Gauss map 
if it exists. The computations were done by using elimination theory and Groebner basis, with the help of the free computer algebra system \textbf{Singular}. Our method is as follows, although some others can already be found in the literature, see for example \cite{macaulay2}.
\\

Suppose that the Gauss map:

\begin{align*}
\Phi: \mathbb{P}^n &\dashrightarrow \mathbb{P}^n\\
(z_0,...,z_n) &\mapsto (A_0,...,A_n),
\end{align*}

\noindent associated to the foliation $\omega=\sum_{i=0}^n A_i(z_0,...,z_n)dz_i$ is birational. Consider the ideal $\mathcal{I}$ in $\C[z_0,...,z_n,a_0,...,a_n]$ generated by the polynomials:

\begin{equation*}
a_0-A_0(z_0,...,z_n),...,a_n-A_n(z_0,...,z_n).
\end{equation*}

\noindent  The Groebner's bases elimination theorem (see \cite{cox}) says that if we fix a lex monomial order where $z_0 > ...> z_n > a_0 > ... > a_n$, then: if $\mathcal{G}$ is a Groebner basis of $\mathcal{I}$, we have that 
for $0 \leq j < n$, the  set $\mathcal{G}_j:=\mathcal{G} \cap \C[z_{j+1},...,z_n,a_0,...,a_n]$ is a Groebner basis for the ideal $\mathcal{I}_j:=\mathcal{I} \cap \C[z_{j+1},...,z_n,a_0,...,a_n]$.
\\

If we consider a polynomial in $\mathcal{G}_{n-1}$ equal to zero and it is possible to isolate the variable $z_n$ (or any other), then we obtain $z_n=B_n(a_0,...,a_n)$. Therefore we substitute $z_n$ in some polynomial in $\mathcal{G}_{n-2}$ to obtain $z_{n-1}=B_{n-1}(a_0,...,a_n)$. Following this process we can obtain the inverse function:

\begin{align*}
\Phi^{-1}: \mathbb{P}^n &\dashrightarrow \mathbb{P}^n\\
(a_0,...,a_n) &\mapsto (B_0,...,B_n).
\end{align*}

\subsection{The exceptional component $\E(2,3;3)$.}\label{ss23'}
In homogeneous coordinates $(x,y,z,t)$, the generic exceptional foliation in $\E(2,3;3)$ is defined by the form
\begin{equation}\label{eq3}
\om=t(2y^2-3xz)dx+t(3tz-xy)dy+t(x^2-2yt)dz+(2x^2z-yzt-xy^2)dt
\end{equation}

In affine coordinates $(t=1)\simeq\C^3$ the Gauss map $\G_\E$ can be written as:

\begin{align*}
\G_\E: \C^3 &\dashrightarrow \C^3\\
(x,y,z) &\mapsto \left( \frac{2y^2-3xz}{\Delta}, \frac{3z-xy}{\Delta}, \frac{x^2-2y}{\Delta}\right)=(a,b,c),
\end{align*}

\noindent where $\Delta(x,y,z)=2x^2z-yz-xy^2$. As we have seen in example \ref{ex4}, this map is birational. Its inverse is:

\begin{align*}
\G_\E^{-1}: \C^3 &\dashrightarrow \C^3\\
(a,b,c) &\to \left( \frac{ab-3c}{2ac-b^2}, \frac{3b-2a^2}{2ac-b^2}, \frac{a^2b+ac-2b^2}{c(2ac-b^2)}\right),
\end{align*}

\noindent the associated foliation in $\mathbb{P}^3$, expressed in homogeneous coordinates $(a,b,c,e)\in\C^4$ is given by the 1-form:

$$\wt\omega=c(ab-3ce)da+c(3be-2a^2)db+(a^2b+ace-2b^2e)dc+c(2ac-b^2)de.$$

$\wt\omega$ is conjugated to $\omega$, with the isomorphism:

$$\Phi(a,b,c,e)=(b,a,e,c)=(x,y,z,w).$$

\noindent This is an example of a self-dual foliation.

\begin{prop}
In homogeneous coordinates, we have:
\[
\G_\E^{-1} \circ \G_\E=t^2(2yt-x^2)(-3x^2y^2+6x^3z+8y^3t-18xyzt+9z^2t^2)Id_{\mathbb{CP}^3}.
\]

In particular, $\G_\E$ contracts the reduced hypersurface 
$$\mathbb{V}(t(2yt-x^2)(-3x^2y^2+6x^3z+8y^3t-18xyzt+9z^2t^2)).$$
\end{prop}

We would like to observe that all these three hypersurfaces are $\E$-invariant.
In fact, they are branches of levels of the first integral $F=f^2/g^3$, where $f(x,y,z,t)=zt^2-xyt+x^3/3$ and $g(x,y,z,t)=yt-x^2/2$. 
\vskip.1in
Since $\E(2,3;n)$ is the set of linear pull-back foliations of $\E(2,3;3)$ by a linear map $L\colon\p^n\dasharrow\p^3$,
by remark \ref{r12}, if $\fa\in \E(2,3;n)$ and $n\ge4$ then $\G_\fa$ cannot be dominant.

\vskip.2in

\subsection{Gauss map for foliations in the boundary of the exceptional component $\E(2,3;3)$}\label{ss}

\subsubsection{The boundary of $\E(2,3;3)$}

As we saw in example \ref{ex4}, the component $\E(2,3;3)$ is the closure of the orbit of one foliation by the action of $SL_4(\mathbb{C})$. In particular, if $\E$ is the foliation defined by $\om$ like in example \ref{ex4} then (see \cite{cl}):

\begin{equation}\label{eq3}
\E(2,3;3)=\ov{\{g^*(\E_\om)\,|\,g\in\, SL_4(\C)\}}\,.
\end{equation} 
\vskip.1in
\noindent Here we will use this fact to determine some representative elements of the boundary of $\E(2,3;3)$. First, we must recall that a 1-parameter subgroup of $SL_4(\C)$ is an algebraic morphism from $\mathbb{C}^*$ to the group and, since $\mathbb{C}^*$ is abelian, all such subgroups are diagonalizable. Then, a 1-parameter subgroup of $SL_4(\C)$ is conjugated to one given as below:
 
\begin{align*}
\lambda_{(n_1,n_2,n_3)}: \C^* &\to SL_4(\C)\\
 s  &\mapsto \left(\begin{array}{cccc}
s^{n_1}&0&0&0 \\
0&s^{n_2}&0&0\\
0&0&s^{n_3}&0\\
0&0&0&s^{n_4}
\end{array}\right),
\end{align*}
\noindent where $n_1, n_2, n_3, n_ 4 \in \Z$ and $n_1+n_2+n_3+n_4=0$.

This 1-parameter subgroup acts on the variables as: 
\[
\lambda_{(n_1,n_2n_3)}(s)(x,y,z,t)=(s^{n_1}x,s^{n_2}y,s^{n_3}z,s^{n_4}t)
\]

In order to see how $\la_{(n_1,n_2,n_3)}$  acts on 
$\E$ we set $\omega=\omega_1+\omega_2+\omega_3$, where

\begin{align*}
\omega_1&=y \big(2yt dx-xtdy-xydt \big)\\
\omega_2&=x \big(-3ztdx+xtdz+2xzdt\big)\\
\omega_3&=t \big(3tz dy-2ytdz-yzdt \big),
\end{align*}

\noindent It can be seen directly that:

\begin{align*}
\lambda_{(n_1,n_2,n_3)}(s) \cdot \omega_1&=s^{n_2-n_{3}}\omega_1,\\
\lambda_{(n_1,n_2,n_3)}(s) \cdot \omega_2&=s^{n_1-n_2}\omega_2,\\
\lambda_{(n_1,n_2,n_3)}(s) \cdot \omega_3&=s^{n_4-n_1}\omega_3.\\
\end{align*}

\noindent If we take $\lambda_1(s)=\lambda_{(1,1,-3)}(s)$, $\lambda_2(s)=\lambda_{(1,-1,-1)}(s)$ and $\lambda_3(s)=\lambda_{(-1,-1,-1)}(s)$ then we have $\lim_{s \to 0} \lambda_i(s) \cdot \omega=\omega_j+\omega_k$, where $i,j,k \in \{1,2,3\}$ are all different from each other.

From the above, we conclude that the forms $\omega_1+\omega_2$, $\omega_2+\omega_3$ and $\omega_1+\omega_3$ are integrable and the foliations associated to them are in the boundary of $\E(2,3;3)$. Let $\E_{ij}$ the foliation defined in homogeneous coordinates by $\om_i+\om_j$, $1\le i<j\le 3$. Then $\pa \E(2,3;3)\supset \bigcup_{i,j}E_{ij}$, where
\[
E_{ij}=\ov{\{g^*(\E_{ij})\,|\,g\in SL_4(\C)\}}
\]

Next we study the foliations in $E_{ij}$, $1\le i<j\le 3$. We will see that the Gauss map of these three foliations are birational.

\begin{enumerate}

\item The first one $\E_{12}$ is defined by:
\begin{align*}
 \omega_1+\omega_2&=(2y^2t-3xzt)dx-xyt\,dy+x^2tdz+(2x^2z-xy^2)dt=\\
&xt(2xz-y^2)\Big(-2\frac{dx}{x}+\frac{dt}{t}+\frac{1}{2}\frac{d(2xz-y^2)}{2xz-y^2}\Big).   
\end{align*}

\noindent The singular set is:

\begin{equation*}
\mathbb{V}(x,y) \cup \mathbb{V}(x,t) \cup \mathbb{V}(t,2xz-y^2).
\end{equation*}

\noindent Its Gauss map at the chart $t=1$ is: 

\begin{align*}
\Phi_{12}: \C^3 &\dashrightarrow \C^3\\
(x,y,z) &\to \Big( \frac{2y^2-3xz}{\Delta_{12}}, \frac{-xy}{\Delta_{12}}, \frac{x^2}{\Delta_{12}}\Big)=(a,b,c),
\end{align*}

\noindent where $\Delta_{12}=2x^2z-xy^2$; its inverse satisfies:

\begin{align*}
cy+bx&=0\\
(2ac^2-b^2c)z-ac+2b^2&=0\\
2by+3cz+ax&=0
\end{align*}

\noindent Solving the above system, we get the inverse:

\begin{align*}
\Phi_{12}^{-1}: \C^3 &\dashrightarrow \C^3\\
(a,b,c) &\to \left( \frac{-3c}{2ac-b^2}, \frac{3b}{2ac-b^2}, \frac{ac-2b^2}{c(2ac-b^2)}\right),
\end{align*}

\noindent The associated foliation in $\mathbb{P}^3$, using homogeneous coordinates $(a,b,c,e)$, is given by the 1-form:

\begin{align*}
  (\omega_1+\omega_2)'&=-3c^2e\,da+3bce\,db+(ace-2b^2e)\,dc+(2ac^2-b^2c)\,de\\
  &=ce(2ac-b^2)\Big(2\frac{dc}{c}+\frac{de}{e}-\frac{3}{2}\frac{d(2ac-b^2)}{2ac-b^2}\Big).
\end{align*}

\noindent Its singular set is:

\begin{equation*}
\mathbb{V}(b,c) \cup \mathbb{V}(c,e) \cup \mathbb{V}(e,2ac-b^2).
\end{equation*}

\begin{prop} 
In homogeneous coordinates, we have:

\begin{align*}
\Phi_{12}^{-1} \circ \Phi_{12}=t^2x^4(-y^2+2xz)Id_{\mathbb{CP}^3},
\end{align*}

\noindent In particular, the reduced hypersurface that $\Phi_{12}$ contracts is 

$$\mathbb{V}(tx(-y^2+2xz)).$$
\end{prop}

\item The second one is:

$$\omega_1+\omega_3=2y^2t\,dx+(3t^2z-xyt)\,dy-2yt^2\,dz-(xy^2+yzt)\,dt,$$

\noindent its singular set is:

\begin{equation*}
\mathbb{V}(x,t) \cup \mathbb{V}(y,t) \cup \mathbb{V}(y,z).
\end{equation*}

The Gauss map at the chart $z=1$ is:

\begin{align*}
\Phi_{13}: \C^3 &\dashrightarrow \C^3\\
(x,y,t) &\to \left(\frac{-y}{t},\frac{xy-3t}{2yt},\frac{xy+t}{2t^2} \right),
\end{align*}

\noindent the inverse is:

\begin{align*}
\Phi_{13}^{-1}: \C^3 &\dashrightarrow \C^3\\
(a,b,e) &\to \left(\frac{ab-3e}{a(e+ab)},\frac{-2a}{e+ab},\frac{2}{e+ab} \right),
\end{align*}

\noindent the associated foliation in $\mathbb{P}^3$, in homogeneous coordinates $(a,b,c,e)$, is given by the 1-form:

$$(\omega_1+\omega_3)'=c(ab-3ec)\,da-2a^2cdb+a(ec+ab)\,dc+2ac^2\,de$$

\noindent Its singular set is:

\begin{equation*}
\mathbb{V}(a,e) \cup \mathbb{V}(a,c) \cup \mathbb{V}(b,c).
\end{equation*}

\begin{prop}
In homogeneous coordinates, we have:

\begin{align*}
\Phi_{13}^{-1} \circ \Phi_{13}=y^4t^4Id_{\mathbb{CP}^3},
\end{align*}

\noindent  In particular, the reduced hypersurface that $\Phi_{13}$ contracts is $\mathbb{V}(yt)$.
\end{prop}

\item The third one is defined by:

$$\omega_2+\omega_3= -3xzt\, dx+3t^2z\,dy+(x^2t-2yt^2)\,dz+(2x^2z-yzt)\,dt$$

\noindent Its singular set is:

\begin{equation*}
\mathbb{V}(x,t) \cup \mathbb{V}(z,t) \cup \mathbb{V}(z,x^2-2yt).
\end{equation*}

The Gauss map at the chart $t=1$ is:

\begin{align*}
\Phi_{23}: \C^3 &\dashrightarrow \C^3\\
(x,y,z) &\to \left(\frac{3x}{y-2x^2},\frac{3}{2x^2-y},\frac{x^2-2y}{z(2x^2-y)} \right)=(a,b,c).
\end{align*}

\noindent The inverse is:

\begin{align*}
\Phi_{23}^{-1}: \C^3 &\dashrightarrow \C^3\\
(a,b,c) &\to \Big(-\frac{a}{b},\frac{2a^2-3b}{b^2}, \frac{-a^2+2b}{cb}\Big)
\end{align*}

\noindent and the associated foliation in $\mathbb{P}^3$, in homogeneous coordinates $(a,b,c,e)$, is given by the 1-form:

$$(\omega_2+\omega_3)'=(-abc)\, da+c(2a^2-3be)\,db+b(-a^2+2be)\,dc+(b^2c)\,de$$

\noindent Its singular set is:

\begin{equation*}
\mathbb{V}(b,c) \cup \mathbb{V}(b,a) \cup \mathbb{V}(c,2bc-a^2).
\end{equation*}

\begin{prop}
In homogeneous coordinates, we have:

\begin{align*}
\Phi_{23}^{-1} \circ \Phi_{23}=z^2t^4(-x^2+2yt)Id_{\mathbb{CP}^3}.
\end{align*}

\noindent In particular, the reduced hypersurface that $\Phi_{23}$ contracts is 

$$\mathbb{V}(zt(-x^2+2yt)).$$
\end{prop}

\end{enumerate}

As a consequence, we can deduce that if $(\a_1,\a_2,\a_3)\in\C^3$ and at least one of the $\a_{i's}$ is not zero, then the Gauss map of the foliation defined by $\a_1\,\om_1+\a_2\,\om_2+\a_3\,\om_3$ is birational.

Another interesting thing about these foliations is that they are in the logarithmic component $\mathcal{L}(1,1,2;3)$, because:

\begin{align*}
\omega_1+\omega_2=&xt(2xz-y^2)\Big(-2\frac{dx}{x}+\frac{dt}{t}+\frac{1}{2}\frac{d(2xz-y^2)}{2xz-y^2}\Big),\\
\omega_2+\omega_3=&zt(x^2-2yt)\Big(2\frac{dt}{t}+\frac{dz}{z}-\frac{3}{2}\frac{d(x^2-2yt)}{x^2-2yt}\Big),\\
\omega_1+\omega_3=&yt(xy-zt)\Big(-3\frac{dy}{y}-\frac{dt}{t}+2\frac{d(xy-zt)}{xy-zt}\Big).
    \end{align*}

\noindent As we saw before, they have birational Gauss map but they are not in the logarithmic components where the generic foliation has birational Gauss map (see corollary \ref{c15}). With this, we can also see that the divisors contracted by the three Gauss maps of the three foliations are invariant to the respective foliations.
\\

Another result is the following:

\begin{prop} We have that $\omega_1+\omega_3$ is a self-dual foliation.
    The foliation $\omega_1+\omega_2$ is projectively equivalent to $(\omega_2+\omega_3)'$; and the foliation $\omega_2+\omega_3$ is projectively equivalent to $(\omega_1+\omega_2)'$ .
\end{prop}

\begin{proof} For the proof we must observe the following calculation, where we have the logarithmic form of the foliations:
\begin{align*}
(\omega_1+\omega_2)'=&ce(2ac-b^2)\Big(2\frac{dc}{c}+\frac{de}{e}-\frac{3}{2}\frac{d(2ac-b^2)}{2ac-b^2}\Big),\\
(\omega_2+\omega_3)'=&bc(a^2-2be)\Big(-2\frac{db}{b}+\frac{dc}{c}+\frac{1}{2}\frac{d(a^2-2be)}{a^2-2be}\Big),\\
(\omega_1+\omega_3)'=&ac(ce-ab)\Big(-3\frac{da}{a}-\frac{dc}{c}+2\frac{d(ce-ab)}{ce-ab}\Big).&
    \end{align*}

    \noindent If we compare with the logarithmic form of the original foliations (that we wrote before); we have the result.
\end{proof}

One final thing we notice with this component is the following: we have found in the boundary of the component $\E(2,3;3)$ a foliation whose Gauss map is not birational, since it is a foliation that is a linear pull-back from a foliation on $\p^2$. This tells us that the property of having a birational Gauss map is not necessarily preserved in the limits. For the construction of this foliation, we consider the 1-parameter subgroup $\beta$ as follows:

$$(x,y,z,t) \mapsto (sx,sy,(s-s^{-3})y+s^{-3}z,st), \quad s \in \mathbb{C}^*,$$

\noindent now we apply this to the foliation given by $\omega$:

\begin{align*}
\beta(s) \cdot \omega&=((-3s^4+3)xyt+2s^4y^2t-3xzt)dx\\
&+((s^4-1)x^2t-s^4xyt+(s^4-1)yt^2+3zt^2)dy\\&+(x^2t-2yt^2)dz+((2s^4-2)x^2y-s^4xy^2+2x^2z+(-s^4+1)y^2t-yzt)dt,
\end{align*}

\noindent therefore we can conclude that:

$$\lim_{s \to \infty} \frac{\beta(s)\cdot \omega}{s^4}=yt(-3x+2y)dx+t(x^2-xy+yt)dy+y(2x^2-xy-yt)dt,$$
\noindent is in the closure of the orbit. It is easy to see that this foliation has singular set of dimension 1 and it is a linear pull-back from a foliation on $\p^2$ of degree 2. Then
its Gauss map is not dominant.

\subsection{The exceptional component $\E(3,4;4)$.}\label{ss33}

As we have said in example \ref{e34} the foliation $\E'$ on $\p^4$ having the first integral
\[
F(s,x,y,z,t)=\frac{\big(st^3-(2xz+y^2)t^2+2yz^2t-\frac{1}{2}z^4\big)^3}{\big(xt^2-yzt-\frac{1}{3}z^3\big)^4}:=\frac{f^3}{g^4}
\]
has birational Gauss map.
This section is dedicated to exposing the details.
As before, the computations were done by using elimination theory and Groebner basis.

First of all, the foliation $\E'$ is defined in homogeneous coordinates by the form 
\[
(*)\,\hskip.3in\om=t^{-2}\,fg\frac{dF}{F}=t^{-2}(3g\,df-4f\,dg):=A\,ds+B\,dx+C\,dy+E\,dz+F\,dt\,,
\]
where:

\begin{align*}
&A=t(z^3-3yzt+3xt^2)\\
&B=-2t(yz^2-2y^2t-xzt+2st^2)\\
&C=-2t(-y^2z+xz^2+3xyt-2szt)\\
&E=-2t(2y^3-5xyz+2sz^2+3x^2t-2syt)\\
&F=2y^3z-6xyz^2+3sz^3+2xy^2t+4x^2zt-5syzt+sxt^2\\
\end{align*}

The inverse of the Gauss map $\G_{\E'}(s,x,y,z,t)=(a,b,c,e,f)$ is\\ $\G_{\E'}^{-1}(a,b,c,e,f)=(S,X,Y,Z,T)$, where

\begin{align*}
S&=-bc^3+3b^2ce-2ac^2e-4abe^2-3b^3f+10abcf-4a^2ef\\
X&=a(2c^3-5bce+6ae^2+4b^2f-8acf)\\
Y&=a(-8abf+b^2e+2bc^2-3bc^2+6ace)\\
Z&=a(16a^2f-2abe-4ac^2+b^2c)\\
T&=a(6abc-12a^2e-b^3).\\
\end{align*}

Finally, we have ${\G}_{\E'}^{-1} \circ \mathcal{G}_{\E'}=t^3.h_3.\,h_9\,Id_{\p^4}$, where
$h_3=3xt^2-3yzt+z^3$ and \\
\begin{align*}
&h_9=8s^3t^6-48s^2xzt^5-24s^2y^2t^5+48s^2yz^2t^4-12s^2z^4t^3+96sx^2z^2t^4+96sxy^2zt^4\\
&-192sxyz^3t^3+48sxz^5t^2+24sy^4t^4-96sy^3z^2t^3+120sy^2z^4t^2-48syz^6t+6sz^8\\
&+81x^4t^5-324x^3yzt^4+44x^3z^3t^3
+390x^2y^2z^2t^3-132x^2yz^4t^2+6x^2z^6t-48xy^4zt^3\\
&-132xy^3z^3x_4^2+84xy^2z^5x_4-12xyz^7-8y^6t^3+48y^5z^2t^2-27y^4z^4t+4y^3z^6\\
\end{align*}

This implies that $\G_{\E'}$ contracts three hypersurfaces $t$, $h_3$ and $h_9$.
All these hypersurfaces are irreducible and $\E'$-invariant.

\subsection{Exceptional examples on $\p^n$, $n\ge5$.}\label{ssexc}
Here we generalize the exceptional components $\E(2,3;3)\sub\fol(2;3)$ and $\E(3,4;4)\sub\fol(3;4)$.
A typical foliation on this component has a rational first integral of the form $f^{n-1}/g^n$, where $f$ and $g$ are irreducible and $deg(f)=n$ and $deg(g)=n-1$.
It will be generated by a Lie algebra of dimension $n$ of linear vector fields on $\C^{n+1}$ 

\begin{lemma}\label{l31}
Let $R=\sum_{j=0}^{n}z_j\frac{\pa}{\pa z_j}$ (radial vector field), $S=\sum_{j=1}^nj\,z_j\frac{\pa}{\pa z_j}$ and $N=\sum_{j=1}^nz_{j-1}\frac{\pa}{\pa z_j}=z_0\frac{\pa}{\pa z_1}+...+z_{n-1}\frac{\pa}{\pa z_n}$.

We will use the notation $N^j$ for the vector field associated to the nilpotent matrix:

\begin{align*}
\left[
\begin{matrix}
0&0&0&...&0&0\\
1&0&0&...&0&0\\
0&1&0&...&0&0\\
.&.&.&...&.&.\\
.&.&.&...&.&.\\
0&0&0&...&0&0\\
0&0&0&...&1&0\\
\end{matrix}
\right]^j.
\end{align*}

We assert that $\G:=\left<R,S,N,N^2,...,N^n\right>_\C$ is a Lie algebra and $\G_k:=\left<R,S,N,N^2,...,N^{k}\right>_\C$ is a Lie sub-algebra of $\G$, for all $k\ge1$.

The generators of $\G$ satisfy the following relations:
\begin{align*}
&[R,S]=[R,N^k]=0\,,\,\forall k\ge1\\
&[S,N^k]=-kN^k\,,\,\forall k\ge1\\
&[N^i,N^j]=0\,,\,\forall i,j\ge1\\
\end{align*}
\end{lemma}

{\it Proof.}
A direct computation shows that $[S,N]=-N$ and that $[R,L]=0$ for all linear vector field $L\in \L(n+1,\C)$.
On the other hand, from $SN-NS=[S,N]=-N$ we get $SN=NS-N$ and $NS=SN+N$. 
From this we get
\[
[S,N^2]=SN^2-N^2S=(NS-N)N-N(SN+N)=-2N^2\,.
\]
If $k\ge2$ then
\[
[S,N^{k+1}]=SN^{k+1}-N^{k+1}S=(SN)N^k-N^k(NS)=(NS-N)N^k-N^k(SN+N)=
\]
\[
=NSN^k-N^kSN-2N^{k+1}=N(SN^{k-1}-N^{k-1}S)N-2N^{k-1}=
\]
\[
=N[S,N^{k-1}]N-2N^{k+1}
\]

Suppose by induction on $j\ge1$ that $[S,N^j]=-jN^j$, $\forall 1\le j\le k$, where $k\ge2$. Then the above computation implies that
\[
[S,N^{k-1}]=-(k-1)N^{k-1}\,\implies\,[S,N^{k+1}]=-(k+1)N^{k+1}\,.
\]
This proves that $\G$ is a Lie algebra and that $\G_k$ is a Lie sub-algebra of $\G$, $\forall k\ge1$.
\qed
\vskip.1in
A consequence of lemma \ref{l31} is that the 1-form 
\[
\om=i_Ri_Si_N...i_{N^{n-2}}\nu\,,\,\text{where}\,\,\nu=dz_1\wedge dz_2\wedge...\wedge dz_{n+1}\,,
\]
is integrable: $\om\wedge d\om=0$. Moreover, $\om$ defines a codimension one foliation $\fa_\om$, of degree $n-1$ of $\p^n$, because their coefficients are homogeneous of degree $n$.
The foliation $\fa_\om$ is generated on homogenous coordinates by the Lie algebra $\G_{n-2}$.

\begin{prop}
The Gauss map of the foliation $\fa_\om$ is birrational.
Moreover it has a rational first integral of the form $f^{n-1}/g^n$, where $deg(f)=n$ and $deg(g)=n-1$.
\end{prop}

{\it Proof.}
Consider the n-vectors on $\C^{n+1}$:
\begin{align*}
&\te= S\wedge N\wedge N^2\wedge...\wedge N^{n-1}:=f.\,\mu\\
&\rho= S\wedge N\wedge N^2\wedge...\wedge N^{n-2}\wedge \frac{\pa}{\pa z_n}:=g.\,\mu\\
\end{align*}
where $\mu=\frac{\pa}{\pa z_1}\wedge...\wedge\frac{\pa}{\pa z_n}$.
Note that $f$ is homogeneous of degree $n$ and $g$ of degree $n-1$.
We assert that $S(f)=n\,f$, $S(g)=(n-1)\,g$ and $N(f)=N(g)=0$. We will use that
\begin{equation}\label{eqw}
L_Y(X_1\wedge...\wedge X_n)=\sum_{j=1}^nX_1\wedge...\wedge [Y,X_j]\wedge...\wedge X_n\,.
\end{equation}

Relation (\ref{eqw}) implies that $L_S(\mu)=-tr(S)\mu$, $tr(S)=1+2+...+n$, so that:
\[
L_S(\te)=S(f)\mu+fL_S(\mu)=S(f)\mu-tr(S)f\mu
\]

On the other hand, by lemma \ref{l31} we have:
\[
L_S\te=\sum_{j=1}^{n-1}S\wedge N\wedge...\wedge [S,N^j]\wedge...\wedge N^{n-1}
=-[tr(S)-n)]\,f\,\mu\,\,\implies
\]
$S(f)=n\,f$. 

In the case of $g$, we have $L_S(\rho)=S(g)\mu-tr(S)g\mu$ and from lemma \ref{l31}:
\[
L_S\left(S\wedge N\wedge N^2\wedge...\wedge N^{n-2}\wedge \frac{\pa}{\pa z_n}\right)=-[tr(S)-(n-1)]\rho\,\,\implies
\]
$S(g)=(n-1)\,g$. This implies
\[
\frac{S(f^{n-1}/g^n)}{f^{n-1}/g^n}=0\,\,\implies
\]
$f^{n-1}/g^n$ is a first integral of $S$. Since $[N,\frac{\pa}{\pa z_0}]=[N,\frac{\pa}{\pa z_n}]=0$, from (\ref{eqw}) we get $N(f)=N(g)=0$, so that $N(f^{n-1}/g^n)=0$.
Finally, since $deg(f^{n-1})=deg(g^n)=n(n-1)$ we get $R(f^{n-1}/g^n)=0$ and $f^{n-1}/g^n$ is a first integral of all the generators of $\G_{n-2}$, and so of the form $\om$.

\begin{rem}
\rm Note that $f=det(A)$ and $g=det(B)$, where $A$ is the $n\times n$ matrix:
\[
(*)\,\,\,A=\left[
\begin{matrix}
z_1&2z_2&3z_3&...&(n-1)z_{n-1}&nz_n\\
z_0&z_1&z_2&...&z_{n-2}&z_{n-1}\\
0&z_0&z_1&...&z_{n-3}&z_{n-2}\\
.&.&.&...&.&.\\
.&.&.&...&.&.\\
0&0&0&...&z_1&z_2\\
0&0&0&...&z_0&z_1\\
\end{matrix}
\right]
\]
and $B$ is the $(n-1)\times(n-1)$ matrix:
\[
(**)\,\,\,B=\left[
\begin{matrix}
z_1&2z_2&3z_3&...&(n-2)z_{n-2}&(n-1)z_{n-1}\\
z_0&z_1&z_2&...&z_{n-3}&z_{n-2}\\
0&z_0&z_1&...&z_{n-4}&z_{n-3}\\
.&.&.&...&.&.\\
.&.&.&...&.&.\\
0&0&0&...&z_1&z_2\\
0&0&0&...&z_0&z_1\\
\end{matrix}
\right]
\]

Using $(*)$ and $(**)$ we obtain that $f(0,z_1,...,z_n)=z_1^n$ and $g(0,z_1,...,z_n)=z_1^{n-1}$, so that
$(z_0=0)$ is contained in the leaf $f^{n-1}-g^n=0$.
\end{rem}

Now, let
\[
\eta=fg\frac{d\left(f^{n-1}/{g^n}\right)}{f^{n-1}/g^n}=(n-1)g\,df-nf\,dg\,.
\]
Since $f^{n-1}/g^n$ is a first integral of $\fa_\om$ we must have $\eta=h.\,\om$, where $h$ is homogeneous polynomial of degree $deg(\eta)-deg(\om)=(2n-2)-n=n-2\ne0$.

\begin{lemma}\label{l32}
We have $h(z)=\a\, z_0^{n-2}$, where $\a=\pm n(n-1)$. 
\end{lemma}

{\it Proof.}
The idea is to prove that in the affine chart $(z_0=1)$ we have $\eta|_{(z_0=1)}=n(n-1)\,\om|_{(z_0=1)}$, which implies $h(z)=\pm n(n-1)z_0^{n-2}$.

First of all, we note that
\[
(\star)\,\om=det\left[
\begin{matrix}
dz_0&dz_1&dz_2&...&dz_{n-2}&dz_{n-1}&dz_n\\
z_0&z_1&z_2&...&z_{n-2}&z_{n-1}&z_n\\
0&z_1&2z_2&...&(n-2)z_{n-2}&(n-1)z_{n-1}&nz_n\\
0&z_0&z_1&...&z_{n-3}&z_{n-2}&z_{n-1}\\
0&0&z_0&...&z_{n-4}&z_{n-3}&z_{n-2}\\
.&.&.&...&.&.&.\\
0&0&0&...&z_1&z_2&z_3\\
0&0&0&...&z_0&z_1&z_2\\
\end{matrix}
\right]:=\sum_{j=0}^nA_jdz_j
\]

Explicitly, we have
\[
A_{n}=(-1)^{(n-1)}z_0\,det\left[
\begin{matrix}
z_1&2z_2&...&(n-2)z_{n-2}&(n-1)z_{n-1}\\
z_0&z_1&...&z_{n-3}&z_{n-2}\\
0&z_0&...&z_{n-4}&z_{n-3}\\
.&.&...&.&.\\
0&0&...&z_1&z_2\\
0&0&...&z_0&z_1\\
\end{matrix}
\right]\,\implies
\]
$\frac{\pa A_n}{\pa z_{n-1}}=(-1)^{n-1}(n-1)z_0^{n-1}$.
We have also
\[
A_{n-1}=(-1)^{(n-2)}z_0\,det\left[
\begin{matrix}
z_1&2z_2&...&(n-2)z_{n-2}&nz_n\\
z_0&z_1&...&z_{n-3}&z_{n-1}\\
0&z_0&...&z_{n-4}&z_{n-2}\\
.&.&...&.&.\\
0&0&...&z_1&z_3\\
0&0&...&z_0&z_2\\
\end{matrix}
\right]\,\implies
\]
$\frac{\pa A_{n-1}}{\pa z_{n}}=(-1)^{n-2}nz_0^{n-1}$.
In particular, we get
\[
\om|_{(z_0=1)}=(-1)^{(n-1)}[-nz_n\,dz_{n-1}+(n-1)z_{n-1}\,dz_n]+\sum_{0<j<n-1}A_j(1,z_1,...,z_n)dz_j\,.
\]

On the other hand, from $(*)$ and $(**)$ we get

\begin{align*}
&\frac{\pa f}{\pa z_n}(z)=n\,z_0^{n-1}&\implies&f|_{(z_0=1)}=nz_n+\wt{f}(z_1,...,z_{n-1})\\
&\frac{\pa g}{\pa z_{n-1}}(z)=(n-1)z_0^{n-2}&\implies&g|_{(z_0=1)}=(n-1)z_{n-1}+\wt{g}(z_1,...,z_{n-2})\\
\end{align*}

Hence, we have
\[
\eta|_{(z_0=1)}=n(n-1)[-nz_n\,dz_{n-1}+(n-1)z_{n-1}\,dz_n]+\sum_{0<j<n-1}B_j(z)dz_j
\]
Since $\eta=h\,\om$, this implies that $h=\pm n(n-1)z_0^{n-2}$ and the lemma.\qed

\begin{lemma}
The Gauss map of $\fa_\om$ is birational.
\end{lemma}

{\it Proof.}
The idea is to prove that $\fa_\om$ satisfies the hypothesis of Corollary \ref{c14} of Theorem \ref{p4}. 
Let us compute the irreducible components of $Sing_2(\fa_\om)$.
First of all, from Lemma \ref{l32} we have $\om|_{(z_0=1)} =\a^{-1}[(n-1)g\,df-nf\,dg]$, so that
\[
\left.
\begin{matrix}
df\wedge\om|_{(z_0=1)}=\a^{-1}\,nf\,df\wedge dg|_{(z_0=1)}\\
dg\wedge\om|_{(z_0=1)}=-\a^{-1}\,(n-1)g\,df\wedge dg|_{(z_0=1)}\\
\end{matrix}
\right\}
\implies
\]
$Sing(\fa_\om)\cap(z_0=1)=(f=g=0)\cap(z_0=1)$, because $df\wedge dg$ is not zero at $z_0=1$.
Moreover, $Sing(\fa_\om)\cap(z_0=1)$ is a union of (singular) orbits of the action of $\G_{n-2}$.
In fact, it is composed of just one orbit: the orbit of the point $p_o:=(1,0,...,0)\in(z_0=1)$, because 
\[
Sing(\fa_\om)\cap(z_0=1)\cap(z_1=z_2=...=z_{n-2}=0)=\{p_o\}\,,
\]
as the reader can check. In particular, $Sing(\fa_\om)\cap(z_0=1)$ is a connected smooth sub-variety of codimension two of $(z_0=1)\simeq\C^n$. We denote it by $V_1$.
Since $(f=0)$ cuts transversely $(g=0)$ along $(z_0=1)$, the transversal type of $\fa_\om$ along $V_1$ has linear part $(n-1)u\,dv-nv\,du$, so that $\mu(\fa_\om,V_1)=1$.

The varieties $(f=0)$ and $(g=0)$ cut $(z_0=0)$ along $V_2:=(z_0=z_1=0)$, but they are tangent along $V_2$ in order $n-1$.
This implies that, as a scheme, $(f=g=0)$ has $V_2$ appearing with multiplicity $n$, reflecting the higher-order tangency, while $V_1$ appears with multiplicity one. In particular, since $deg(f)=n$ and $deg(g)=n-1$, we have $deg(V_1)=n(n-1)-n=n^2-2n$.

Now, from $(\star)$ we have $Sing(\fa_\om)\cap(z_0=0)=(z_1.\,F=0)$, where
\[
F(z_1,...,z_n)=det\left[
\begin{matrix}
1&z_2&...&z_{n-2}&z_{n-1}&z_n\\
1&2z_2&...&(n-2)z_{n-2}&(n-1)z_{n-1}&nz_n\\
0&z_1&...&z_{n-3}&z_{n-2}&z_{n-1}\\
0&0&...&z_{n-4}&z_{n-3}&z_{n-2}\\
.&.&...&.&.&.\\
0&0&...&z_1&z_2&z_3\\
0&0&...&0&z_1&z_2\\
\end{matrix}
\right]\,\implies
\]
$Sing(\fa_\om)\cap(z_0=0)=V_2\cup V_3$, where $V_3=(z_0=F=0)$ and $deg(V_3)=n-1$.

Finally, since $deg(V_1)+deg(V_2)+deg(V_3)=n^2-n=(n-1)^2+(n-1)=deg(\fa_\om)+deg^2(\fa_\om)$ the Gauss map of $\fa_\om$ is birational, by corollary \ref{c14}.\qed

\vskip.1in

For example, in the case $n=5$ we have:

\begin{align*}
f&=z_1^5-5z_0z_1^3z_2+5z_0^2z_1z_2^2+5z_0^2z_1^2z_3-5z_0^3z_2z_3-5z_0^3z_1z_4+5z_0^4z_5\\
g&=z_1^4-4z_0z_1^2z_2+2z_0^2z_2^2+4z_0^2z_1z_3-4z_0^3z_4.
\end{align*}

\noindent  In particular, the form defining the foliation is $\om=\sum_{j=0}^5A_j(z)\,dz_j$, where:

\begin{align*}
A_0&=z_1z_2^4-4z_1^2z_2^2z_3+z_0z_2^3z_3+2z_1^3z_3^2+2z_0z_1z_2z_3^2+4z_1^3z_2z_4-6z_0z_1z_2^2z_4-5z_0z_1^2z_3z_4+\\ &3z_0^2z_2z_3z_4+3z_0^2z_1z_4^2-4z_1^4z_5+11z_0z_1^2z_2z_5-3z_0^2z_2^2z_5-6z_0^2z_1z_3z_5+z_0^3z_4z_5\\
A_1&=z_0(-2z_2^4+7z_1z_2^2z_3-3z_1^2z_3^2-5z_0z_2z_3^2-6z_1^2z_2z_4+\\
&6z_0z_2^2z_4+7z_0z_1z_3z_4-4z_0^2z_4^2+5z_1^3z_5-10z_0z_1z_2z_5+5z_0^2z_3z_5)\\
A_2&=z_0(z_1z_2^3-2z_1^2z_2z_3-3z_0z_2^2z_3+4z_0z_1z_3^2+z_1^3z_4+3z_0z_1z_2z_4-4z_0^2z_3z_4-5z_0z_1^2z_5+5z_0^2z_2z_5)\\
A_3&=z_0(-z_1^2z_2^2+2z_0z_2^3+z_1^3z_3-z_0z_1z_2z_3-z_0z_1^2z_4-4z_0^2z_2z_4+5z_0^2z_1z_5)\\
A_4&=z_0(z_1^3z_2-3z_0z_1z_2^2-z_0z_1^2z_3+5z_0^2z_2z_3+z_0^2z_1z_4-5z_0^3z_5)\\
A_5&=z_0(-z_1^4+4z_0z_1^2z_2-2z_0^2z_2^2-4z_0^2z_1z_3+4z_0^3z_4).
\end{align*}

The singular set $Sing(\mathcal{F}_\om)$ has three irreducible components, say $V_1$, $V_2$ and $V_3$, where $V_3\cap (z_0=1)$ is $(f|_{z_0=1}=g|_{z_0=1}=0)$. The closure of this variety is irreducible of degree $15$. Finally, $Sing(\fa_\om)\cap(z_0=0)=V_1\cup V_2$, where:

\begin{align*}
&V_1=(z_0=z_1=0)\\
&V_2=(z_0=z_2^4-4z_1z_2^2z_3+2z_1^2z_3^2+4z_1^2z_2z_4-4z_1^3z_5=0)\\
\end{align*}

\noindent the first of degree $1$ and the second of degree $4$.
\\

\vskip.2in

\bibliographystyle{amsalpha}

\end{document}